\documentclass[a4paper,12pt]{amsart}%
\usepackage{amsmath,amsfonts,amssymb,amsthm, comment}
\usepackage{mathrsfs}
\usepackage{mathscinet}
\usepackage{graphicx}
\usepackage{enumerate}
\usepackage{algorithm}
\usepackage[noend]{algpseudocode}
\usepackage{color}
\usepackage[numbers]{natbib}
\usepackage[hyperindex,breaklinks]{hyperref}
\usepackage{xspace}
\usepackage{multirow}
\usepackage{amsmath}
\usepackage{amsfonts}
\usepackage{amssymb}%
 \usepackage[foot]{amsaddr}
\providecommand{\U}[1]{\protect\rule{.1in}{.1in}}
\def\citeapos#1{\citeauthor{#1} (\citeyear{#1})}
\hypersetup{
colorlinks   = true,   urlcolor     = blue,   linkcolor    = blue,   citecolor    = blue }
\providecommand{\U}[1]{\protect\rule{.1in}{.1in}}
\hypersetup{
pdftitle={},
pdfsubject={Probability Theory},
pdfauthor={Jose Blanchet, Karthyek Murthy},
pdfdisplaydoctitle=true,
}
\providecommand{\U}[1]{\protect\rule{.1in}{.1in}}

\newtheorem{theorem}{Theorem}

\newtheorem{proposition}{Proposition}

\theoremstyle{remark}

\begin{document}
\title[ ]{Distributionally Robust Groupwise Regularization Estimator}
\author{Blanchet, J.}
%\address{Columbia University, Department of Statistics and Department of Industrial
%Engineering \& Operations Research, 340 S. W. Mudd Building, 500 W. 120
%Street, New York, NY 10027, United States.}
\email{jose.blanchet@columbia.edu}
\author{Kang, Y.}
\address{Columbia University. New York, NY 10027, United States.}
\email{yang.kang@columbia.edu}

\date{\today }
\maketitle

\begin{abstract}
Regularized estimators in the context of group variables have been applied
successfully in model and feature selection in order to preserve
interpretability. We formulate a Distributionally Robust Optimization (DRO)
problem which recovers popular estimators, such as Group Square Root Lasso
(GSRL). Our DRO formulation allows us to interpret GSRL as a game, in which we
learn a regression parameter while an adversary chooses a perturbation of the
data. We wish to pick the parameter to minimize the expected loss under any
plausible model chosen by the adversary - who, on the other hand, wishes to
increase the expected loss. The regularization parameter turns out to be
precisely determined by the amount of perturbation on the training data
allowed by the adversary. In this paper, we introduce a data-driven
(statistical) criterion for the optimal choice of regularization, which we
evaluate asymptotically, in closed form, as the size of the training set
increases. Our easy-to-evaluate regularization formula is compared against
cross-validation, showing good (sometimes superior) performance.

\end{abstract}

\section{Introduction}

Group Lasso (GR-Lasso) estimator is a generalization of the Lasso estimator
(see \citeapos{tibshirani_regression_1996}). The method focuses on variable
selection in settings where some predictive variables, if selected, must be
chosen as a group. For example, in the context of the use of dummy variables
to encode a categorical predictor, the application of the standard Lasso
procedure might result in the algorithm including only a few of the variables
but not all of them, which could make the resulting model difficult to
interpret. Another example, where the GR-Lasso estimator is particularly
useful, arises in the context of feature selection. Once again, a particular
feature might be represented by several variables, which often should be
considered as a group in the variable selection process. \newline The GR-Lasso
estimator was initially developed for the linear regression case (see
\citeapos{yuan2006model}), but a similar group-wise regularization was also
applied to logistic regression in \citeapos{meier2008group}. A brief summary of
GR-Lasso technique type of methods can be found in \citeapos{friedman2010note}.
\newline Recently, \citeapos{bunea2014group} developed a variation of the GR-Lasso
estimator, called the Group-Square-Root-Lasso (GSRL) estimator, which is very
similar to the GR-Lasso estimator. The GSRL is to the GR-Lasso estimator what
sqrt-Lasso, introduced in \citeapos{belloni2011square}, is to the standard Lasso
estimator. In particular, GSRL has a superior advantage over GR-Lasso, namely,
that the regularization parameter can be chosen independently from the
standard deviation of the regression error in order to guarantee the
statistical consistency of the regression estimator (see
\citeapos{belloni2011square}, and \citeapos{bunea2014group}). 
\bigskip\newline
 Our contribution in this paper is to provide a DRO representation for the GSRL
estimator, which is rich in interpretability and which provides insights to
optimally select (using a natural criterion) the regularization parameter
without the need of time-consuming cross-validation. We compute the optimal
regularization choice (based on a simple formula we derive in this paper) and
evaluate its performance empirically. We will show that our method for the
regularization parameter is comparable, and sometimes superior, to
cross-validation. 
\bigskip\newline 
In order to describe our contributions more
precisely, let us briefly describe the GSRL estimator. We choose the context
of linear regression to simplify the exposition, but an entirely analogous
discussion applies to the context of logistic regression. \newline Consider a
given a set of training data $\{(X_{1},Y_{1}),\ldots,(X_{n},Y_{n})\}$. The
input $X_{i}\in\mathbb{R}^{d}$ is a vector of $d$ predicting variables, and
$Y_{i}\in\mathbb{R}$ is the response variable. (Throughout the paper any
vector is understood to be a column vector and the transpose of $x$ is denoted
by $x^{T}$.) We use $\left(  X,Y\right)  $ to denote a generic sample from the
training data set. It is postulated that
\[
Y_{i}=X_{i}^{T}\beta^{\ast}+e_{i},
\]
for some $\beta^{\ast}\in\mathbb{R}^{d}$ and errors $\left\{  e_{1}%
,...,e_{n}\right\}  $. Under suitable statistical assumptions (such as
independence of the samples in the training data), one may be interested in
estimating $\beta^{\ast}$. \newline Underlying, we consider the square loss
function, i.e. $l\left(  x,y;\beta\right)  =\left(  y-\beta^{T}x\right)  ^{2}%
$, for the purpose of this discussion but this choice, as we shall see, is not
necessary. 
\bigskip\newline Throughout the paper we will assume the following group
structure for the space of predictors. There are $\bar{d}\leq d$ mutually
exclusive groups, which form a partition. More precisely, suppose that
$G_{1},\ldots,G_{\bar{d}}$ satisfies that $G_{i}\cap G_{j}=\varnothing$ for
$i\neq j$, that $G_{1}\cup...\cup G_{\bar{d}}=\{1,...,d\}$, and the $G_{i}$'s
are non-empty. We will use $g_{i}$ to denote the cardinality of $G_{i}$ and
shall write $G$ for a generic set in the partition and let $g$ denote the
cardinality of $G$. \newline We shall denote by $x\left(  G\right)
\in\mathbb{R}^{g}$ the sub-vector $x\in\mathbb{R}^{d}$ corresponding to $G$.
So, if $G=\{i_{1},...,i_{g}\}$, then $x\left(  G\right)  =\left(  X_{i_{1}%
},\ldots,X_{i_{g}}\right)  ^{T}$. \newline Next, given $p,s\geq1$, and
$\alpha\in\mathbb{R}_{++}^{\bar{d}}$ (i.e. $\alpha_{i}>0$ for $1\leq i\leq
\bar{d}$) we define for each $x\in\mathbb{R}^{d}$,
\begin{equation}
\left\Vert x\right\Vert _{\alpha\text{-}(p,s)}=\left(  \sum_{i=1}^{\bar{d}%
}\alpha_{i}^{s}\left\Vert x\left(  G_{i}\right)  \right\Vert _{p}^{s}\right)
^{1/s}, \label{alpha_ps_norm}%
\end{equation}
where $\left\Vert x\left(  G_{i}\right)  \right\Vert _{p}$ denotes the
$p$-norm of $x\left(  G_{i}\right)  $ in $\mathbb{R}^{g_{i}}$. (We will study
fundamental properties of $\left\Vert x\right\Vert _{\alpha\text{-}(p,s)}$ as
a norm in Proposition \ref{Thm-Dual-Norm}.) \newline Let $P_{n}$ be the
empirical distribution function, namely,%
\[
P_{n}\left(  dx,dy\right)  :=\frac{1}{n}\sum_{i=1}^{n}\mathcal{\delta
}_{\left\{  (X_{i},Y_{i})\right\}  }(dx,dy).
\]
Throughout out the paper we use the notation $\mathbb{E}_{P}[\cdot]$ to denote
expectation with respect to a probability distribution $P$. \newline The GSRL
estimator takes the form
\[
\min_{\beta}\sqrt{\frac{1}{n}\sum_{i=1}^{n}l\left(  X_{i,}Y_{i};\beta\right)
}+\lambda\left\Vert \beta\right\Vert _{\tilde{g}^{-1}-\left(  2,1\right)
}=\min_{\beta}\left(  \mathbb{E}_{P_{n}}^{1/2}\left[  l\left(  X,Y;\beta
\right)  \right]  +\lambda\left\Vert \beta\right\Vert _{\sqrt{\tilde{g}%
}-\left(  2,1\right)  }\right)  ,
\]
where $\lambda$ is the so-called regularization parameter. The previous
optimization problem can be easily solved using standard convex optimization
techniques as explained in \citeapos{belloni2011square} and \citeapos{bunea2014group}.
\bigskip\newline 
Our contributions in this paper can now be explicitly stated. We
introduce a notion of discrepancy, $\mathcal{D}_{c}\left(  P,P_{n}\right)  $,
discussed in Section \ref{Sect_2}, between $P_{n}$ and any other probability
measure $P$, such that%
\begin{equation}
\min_{\beta}\max_{P:D_{c}\left(  P,P_{n}\right)  \leq\delta}\mathbb{E}%
_{P}^{1/2}\left[  l\left(  X,Y;\beta\right)  \right]  =\min_{\beta}\left(
\mathbb{E}_{P_{n}}^{1/2}\left[  l\left(  X,Y;\beta\right)  \right]
+\delta^{1/2}\left\Vert \beta\right\Vert _{\alpha-\left(  p,s\right)
}\right)  . \label{DRO_I}%
\end{equation}
Using this representation, which we formulate, together with its logistic
regression analogue, in Section \ref{Sect_2_2_1_Lin} and Section
\ref{Sect_2_2_2_Log}, we are able to draw the following insights:
\bigskip

\textbf{I)} GSRL can be interpreted as a game in which we choose a parameter
(i.e. $\beta$) and an adversary chooses a \textquotedblleft
plausible\textquotedblright\ perturbation of the data (i.e. $P$); the
parameter $\delta$ controls the degree in which $P_{n}$ is allowed to be
perturbed to produce $P$. The value of the game is dictated by the expected
loss, under $E_{P}$, of the decision variable $\beta$.

\textbf{II)}\ The set $\mathcal{U}_{\delta}\left(  P_{n}\right)  =\left\{
P:\mathcal{D}_{c}\left(  P,P_{n}\right)  \leq\delta\right\}  $ denotes the set
of distributional uncertainty. It represents the set of plausible variations
of the underlying probabilistic model which are reasonably consistent with the data.

\textbf{III)}\ The DRO representation (\ref{DRO_I}) exposes the role of the
regularization parameter. In particular, because $\lambda=\delta^{1/2}$, we
conclude that $\lambda$ directly controls the size of the distributionally
uncertainty and should be interpreted as the parameter which dictates the
degree to which perturbations or variations of the available data should be considered.

\textbf{IV)} As a consequence of I)\ to III), the DRO\ representation
(\ref{DRO_I}) endows the GSRL estimator with desirable generalization
properties. The GSRL aims at choosing a parameter, $\beta$, which should
perform well for \textit{all} possible probabilistic descriptions which are
plausible given the data. 
\bigskip\newline 
Naturally, it is important to
understand what types of variations or perturbations are measured by the
discrepancy $\mathcal{D}_{c}\left(  P,P_{n}\right)  $. For example, a popular
notion of the discrepancy is the Kullback-Leibler (KL) divergence. However, KL
divergence has the limitation that only considers probability distributions
which are supported precisely on the available training data, and therefore
potentially ignores plausible variations of the data which could have an
adverse impact on generalization risk. \newline In the rest of the paper we
answer the following questions. First, in Section \ref{Sect_2} we explain the
nature of the discrepancy $\mathcal{D}_{c}\left(  P,P_{n}\right)  $, which we
choose as an Optimal Transport discrepancy. We will see that $\mathcal{D}%
_{c}\left(  P,P_{n}\right)  $ can be computed using a linear program. \newline
Intuitively, $\mathcal{D}_{c}\left(  P,P_{n}\right)  $ represents the minimal
transportation cost for moving the mass encoded by $P_{n}$ into a sinkhole
which is represented by $P$. The cost of moving mass from location $u=\left(
x,y\right)  $ to $w=\left(  x^{\prime},y^{\prime}\right)  $ is encoded by a
cost function $c\left(  u,w\right)  $ which we shall discuss and this will
depend on the $\alpha$-$\left(  p,s\right)  $ norm that we defined in
(\ref{alpha_ps_norm}). The subindex $c$ in $\mathcal{D}_{c}\left(
P,P_{n}\right)  $ represents the dependence on the chosen cost function.
\newline The next item of interest is the choice of $\delta$, again the
discussion of items I)\ to III)\ of the DRO formulation (\ref{DRO_I}) provides
a natural way to optimally choose $\delta$. The idea is that every model
$P\in\mathcal{U}_{\delta}\left(  P_{n}\right)  $ should intuitively represent
a plausible variation of $P_{n}$ and therefore $\beta^{P}=\arg\min\left\{
\mathbb{E}_{P}[l\left(  X,Y;\beta\right)  ]:\beta\right\}  $ is a plausible
estimate of $\beta^{\ast}$. The set $\{\beta^{P}:P\in\mathcal{U}_{\delta
}\left(  P_{n}\right)  \}$ therefore yields a confidence region for
$\beta^{\ast}$ which is increasing in size as $\delta$ increases. Hence, it is
natural to minimize $\delta$ to guarantee a target confidence level (say
95\%). In Section \ref{Sec-Asymptotic} we explain how this optimal choice can
be asymptotically computed as $n\rightarrow\infty$. \newline Finally, it is of
interest to investigate if the optimal choice of $\delta$ (and thus of
$\lambda$) actually performs well in practice. We compare performance of our
(asymptotically)\ optimal choice of $\lambda$ against cross-validation
empirically in Section \ref{Sec-Numerical}. We conclude that our choice is
quite comparable to cross validation. 
\bigskip\newline Before we continue with the
program that we have outlined, we wish to conclude this Introduction with a
brief discussion of work related to the methods discussed in this paper.
Connections between regularized estimators and robust optimization
formulations have been studied in the literature. For example, the work of
\citeapos{xu_robust_2009} investigates determinist perturbations on the predictor
variables to quantify uncertainty. In contrast, our DRO approach quantifies
perturbations from the empirical measure. This distinction allows us to
statistical theory which is key to optimize the size of the uncertainty,
$\delta$ (and thus the regularization parameter) in a data-driven way. The
work of \citeapos{shafieezadeh-abadeh_distributionally_2015} provides connections
to regularization in the setting of logistic regression in an approximate
form. More importantly, \citeapos{shafieezadeh-abadeh_distributionally_2015}
propose a data-driven way to choose the size of uncertainty, $\delta$, which
is based on the concentration of measure results. The concentration of measure
method depends on restrictive assumptions and leads to suboptimal choices,
which deteriorate poorly in high dimensions. 
\bigskip\newline The present work is a
continuation of the line of research development in
\citeapos{blanchet2016robust}, which
concentrates only on classical regularized estimators without the group
structure. Our current contributions require the development of duality
results behind the $\alpha$-$\left(  p,t\right)  $ norm which closely
parallels that of the standard duality between $l_{p}$ and $l_{q}$ spaces
(with $1/p+1/q=1$) and a number of adaptations and interpretations that are
special to the group setting only.

\section{Optimal Transport and DRO\label{Sect_2}}

\subsection{Defining the optimal transport discrepancy\label{Sect_2_1_OT}}

Let $c:\mathbb{R}^{d+1}\times\mathbb{R}^{d+1}\rightarrow\lbrack0,\infty]$ be
lower semicontinuous and we assume that $c(u,w)=0$ if and only if $u=w$. For
reasons that will become apparent in the sequel, we will refer to $c\left(
\cdot\right)  $ as a cost function. \newline Given two distributions $P$ and
$Q$, with supports $\mathcal{S}_{P}\subseteq\mathbb{R}^{d+1}$ and
$\mathcal{S}_{Q}\subseteq\mathbb{R}^{d+1}$, respectively, we define the
optimal transport discrepancy, $\mathcal{D}_{c}$, via%
\begin{equation}
\mathcal{D}_{c}\left(  P,Q\right)  =\inf_{\pi}\{E_{\pi}\left[  c\left(
U,W\right)  \right]  :\pi\in\mathcal{P}\left(  \mathcal{S}_{P}\times
\mathcal{S}_{Q}\right)  ,\text{ }\pi_{U}=P,\text{ }\pi_{W}=Q\},
\label{Discrepancy_Def}%
\end{equation}
where $\mathcal{P}\left(  \mathcal{S}_{P}\times\mathcal{S}_{Q}\right)  $ is
the set of probability distributions $\pi$ supported on $\mathcal{S}_{P}%
\times\mathcal{S}_{Q}$, and $\pi_{U}$ and $\pi_{W}$ denote the marginals of
$U$ and $W$ under $\pi$, respectively. \newline If, in addition, $c\left(
\cdot\right)  $ is symmetric (i.e. $c\left(  u,w\right)  =c\left(  w,u\right)
$), and there exists $\varrho\geq1$ such that $c^{1/\varrho}\left(
u,w\right)  \leq c^{1/\varrho}\left(  u,v\right)  +c^{1/\varrho}\left(
v,w\right)  $ (i.e. $c^{1/\varrho}\left(  \cdot\right)  $ satisfies the
triangle inequality), it can be easily verified (see
\citeapos{villani_optimal_2008}) that $\mathcal{D}_{c}^{1/\varrho}\left(
P,Q\right)  $ is a metric on the probability measures. For example, if
$c\left(  u,w\right)  =\left\Vert u-w\right\Vert _{q}^{\varrho}$ for $q\geq1$
(where $\left\Vert u-w\right\Vert _{q}$ denotes the $l_{q}$ norm in
$\mathbb{R}^{d+1}$) then $\mathcal{D}_{c}\left(  \cdot\right)  $ is known as
the Wasserstein distance of order $\varrho$. \newline Observe that
(\ref{Discrepancy_Def}) is obtained by solving a linear programming problem.
For example, suppose that $Q=P_{n}$, so $\mathcal{S}_{P_{n}}=\{\left(
X_{i},Y_{i}\right)  \}_{i=1}^{n}$, and let $P\ $ be supported in some finite
set $\mathcal{S}_{P}$ then, using $U=\left(  X,Y\right)  $, we have that
$D_{c}\left(  P,P_{n}\right)  $ is obtained by computing%
\begin{align}
&  \left.  \min_{\pi}\sum_{u\in\mathcal{S}_{P}}\sum_{w\in\mathcal{S}_{P_{n}}%
}c\left(  u,w\right)  \pi\left(  u,w\right)  \right. \label{LP}\\
\text{s.t.}  &  \left.  \sum_{u\in\mathcal{S}_{P}}\pi\left(  u,w\right)
=\frac{1}{n}\text{ }\forall\text{ }w\in\mathcal{S}_{P_{n}}\right. \nonumber\\
&  \left.  \sum_{w\in\mathcal{S}_{P_{n}}}\pi\left(  u,w\right)  =P\left(
\left\{  u\right\}  \right)  \text{ }\forall\text{ }u\in\mathcal{S}_{P}\right.
\nonumber\\
&  \left.  \pi\left(  u,w\right)  \geq0\text{ }\forall\text{ }\left(
u,w\right)  \in\mathcal{S}_{P}\times\mathcal{S}_{P_{n}}\right.  .\nonumber
\end{align}
A completely analogous linear program (LP), albeit an infinite dimensional
one, can be defined if $\mathcal{S}_{P}$ has infinitely many elements. This LP
has been extensively studied in great generality in the context of Optimal
Transport under the name of Kantorovich's problem (see
\citeapos{villani_optimal_2008})). \newline Note that Kantorovich's problem is
always feasible (take $\pi$ with independent marginals, for example).
Moreover, under our assumptions on $c$, if the optimal value is finite, then
there is an optimal solution $\pi^{\ast}$ (see Chapter 1 of
\citeapos{villani_optimal_2008})). \newline It is clear from the formulation of
the previous LP\ that $\mathcal{D}_{c}\left(  P,P_{n}\right)  $ can be
interpreted as the minimal cost of transporting mass from $P_{n}$ to $P$,
assuming that the marginal cost of transporting the mass from $u\in
\mathcal{S}_{P}$ to $w\in\mathcal{S}_{P_{n}}$ is $c\left(  u,w\right)  $. It
is also not difficult to realize from the assumption that $c\left(
u,w\right)  =0$ if and only if $u=w$ that $\mathcal{D}_{c}\left(
P,P_{n}\right)  =0$ if and only if $P=P_{n}$. We shall discuss, for instance,
how to choose $c\left(  \cdot\right)  $ to recover (\ref{DRO_I}) and the
corresponding logistic regression formulation of GR-Lasso.

\subsection{DRO Representation of GSRL Estimators\label{Sect_2_2_DRO}}

In this section, we will construct a cost function $c\left(  \cdot\right)  $
to obtain the GSRL (or GR-Lasso) estimators. We will follow an approach
introduced in \citeapos{blanchet2016robust} for the context of square-root Lasso
(SR-Lasso) and regularized logistic regression estimators.

\subsubsection{GSRL Estimators for Linear Regression\label{Sect_2_2_1_Lin}}

We start by assuming precisely the linear regression setup described in the
Introduction and leading to (\ref{DRO_I}). Given $\alpha=(\alpha
_{1},...,\alpha_{\bar{d}})^{T}\in R_{++}^{\bar{d}}$ define $\alpha
^{-1}=\left(  \alpha_{1}^{-1},...,\alpha_{\bar{d}}^{-1}\right)  ^{T}$. Now,
underlying there is a partition $G_{1},...,G_{\bar{d}}$ of $\{1,...,d\}$ and
given $q,t\in\lbrack1,\infty]$ we introduce the cost function%
\begin{equation}
c\left(  \left(  x,y\right)  ,\left(  x^{\prime},y^{\prime}\right)  \right)
=\left\{
\begin{array}
[c]{ccc}%
\left\Vert x-x^{\prime}\right\Vert _{\alpha^{-1}\text{-}(q,t)}^{\varrho} &
\text{if} & y=y^{\prime}\\
\infty & \text{if} & y\neq y^{\prime}%
\end{array}
\right.  , \label{Def_c_LR}%
\end{equation}
where, following (\ref{alpha_ps_norm}), we have that
\[
\left\Vert x-x^{\prime}\right\Vert _{\alpha^{-1}\text{-}(q,t)}^{\varrho
}=\left(  \sum_{i=1}^{\bar{d}}\alpha_{i}^{-t}\left\Vert x\left(  G_{i}\right)
-x^{\prime}\left(  G_{i}\right)  \right\Vert _{q}^{t}\right)  ^{\varrho/t}.
\]
Then, we obtain the following result.

\begin{theorem}
[DRO Representation for Linear Regression GSRL]%
\label{Thm-Dro-GLasso-Regression} Suppose that $q,t\in\lbrack1,\infty]$ and
$\alpha\in R_{++}^{\bar{d}}$ are given and $c\left(  \cdot\right)  $ is
defined as in (\ref{Def_c_LR}) for $\varrho=2$. Then, if $l\left(
x,y;\beta\right)  =\left(  y-x^{T}\beta\right)  ^{2}$ we obtain
\[
\quad\min_{\beta\in\mathbb{R}^{d}}\sup_{P:D_{c}\left(  P,P_{n}\right)
\leq\delta}\left(  \mathbb{E}_{P}\left[  l\left(  X,Y;\beta\right)  \right]
\right)  ^{1/2}=\min_{\beta\in\mathbb{R}^{d}} \left(  {E_{P_{n}}\left[
l\left(  X,Y;\beta\right)  \right]  }\right)  ^{1/2}+\sqrt{\delta}\left\Vert
\beta\right\Vert _{\alpha\text{-}(p,s)},
\]
where $1/p+1/q=1$, and $1/s+1/t=1$.
\end{theorem}

We remark that choosing $p=q=2$, $t=\infty$, $s=1$, and $\alpha_{i}%
=\sqrt{g_{i}}$ for $i\in\{1,...,\bar{d}\}$ we end up obtaining the GSRL
estimator formulated in \citeapos{bunea2014group}). \newline We note that the cost
function $c\left(  \cdot\right)  $ only allows mass transportation on the
predictors (i.e $X$), but no mass transportation is allowed on the response
variable $Y$. This implies that the GSRL estimator implicitly assumes that
distributional uncertainty is only present on prediction variables (i.e.
variations on the data only occurs through the predictors).

\subsubsection{GR-Lasso Estimators for Logistic
Regression\label{Sect_2_2_2_Log}}

We now discuss GR-Lasso for classification problems. We consider a training
data set of the form $\{(X_{1},Y_{1}),\ldots,(X_{n},Y_{n})\}$. Once again, the
input $X_{i}\in\mathbb{R}^{d}$ is a vector of $d$ predictor variables, but now
the response variable $Y_{i}\in\{-1,1\}$ is a categorical variable. In this
section we shall consider as our loss function the log-exponential function,
namely,
\begin{equation}
l\left(  x,y;\beta\right)  =\log\left(  1+\exp\left(  -y\beta^{T}x\right)
\right)  . \label{E_log_Expo_loss}%
\end{equation}
This loss function is motivated by a logistic regression model which we shall
review in the sequel. But for the DRO representation formulation it is not
necessary to impose any statistical assumption. We then obtain the following theorem.

\begin{theorem}
[DRO Representation for Logistic Regression GR-Lasso]%
\label{Thm-Dro-GLasso-Classification} Suppose that $q,t\in\lbrack1,\infty]$
and $\alpha\in R_{++}^{\bar{d}}$ are given and $c\left(  \cdot\right)  $ is
defined as in (\ref{Def_c_LR}) for $\varrho=1$. Then, if $l\left(
x,y;\beta\right)  $ is defined as in (\ref{E_log_Expo_loss}) we obtain%
\[
\min_{\beta\in\mathbb{R}^{d}}\sup_{P:D_{c}\left(  P,P_{n}\right)  \leq\delta
}\mathbb{E}_{P}\left[  l\left(  X,Y;\beta\right)  \right]  =\min_{\beta
\in\mathbb{R}^{d}}E_{P_{n}}\left(  l\left(  X,Y;\beta\right)  \right)
+\delta\left\Vert \beta\right\Vert _{\alpha\text{-}(p,s)},
\]
where $1\leq q,t\leq\infty$, $1/p+1/q=1$ and $1/s+1/t=1$.
\end{theorem}

We note that by taking $p=q=2$, $t=\infty$, $s=1$, $\alpha_{i}=\sqrt{g_{i}}$
for $i\in\{1,...,\bar{d}\}$, and $\lambda=\delta$ we recover the GR-Lasso
logistic regression estimator from \citeapos{meier2008group}.

As discussed in the previous subsection, the choice of $c\left(  \cdot\right)
$ implies that the GR-Lasso estimator implicitly assumes that distributionally
uncertainty is only present on prediction variables.

\section{Optimal Choice of Regularization Parameter\label{Sec-Asymptotic}}

Let us now discuss the mathematical formulation of the optimal criterion that
we discussed for choosing $\delta$ (and therefore the regularization parameter
$\lambda$). We define
\[
\Lambda_{\delta}\left(  P_{n}\right)  =\{\beta^{P}:P\in\mathcal{U}_{\delta
}\left(  P_{n}\right)  \},
\]
as discussed in the Introduction, $\Lambda_{\delta}\left(  P_{n}\right)  $ is
a natural confidence region for $\beta^{\ast}$ because each element $P$ in the
distributional uncertainty set $\mathcal{U}_{\delta}\left(  P_{n}\right)  $
can be interpreted as a plausible variation of the empirical data $P_{n}$.
Then, given a confidence level $1-\chi$ (say $1-\chi=.95$) we wish to choose
\[
\delta_{n}^{\ast}=\inf\{\delta:P\left(  \beta^{\ast}\in\Lambda_{\delta}\left(
P_{n}\right)  \right)  >1-\chi\}.
\]
Note that in the evaluation of $P\left(  \beta^{\ast}\in\Lambda_{\delta
}\left(  P_{n}\right)  \right)  $ the random element is $P_{n}$. So, we shall
impose natural probabilistic assumptions on the data generating process in
order to \textit{asymptotically} \textit{evaluate} $\delta_{n}^{\ast}$ as
$n\rightarrow\infty$.

\subsection{The Robust Wasserstein Profile Function\label{Sect_2_3_1_RWP}}

In order to asymptotically evaluate $\delta_{n}^{\ast}$ we must recall basic
properties of the so-called Robust Wassertein Profile function (RWP function)
introduced in \citeapos{blanchet2016robust}. \newline Suppose for each $\left(
x,y\right)  $, the loss function $l\left(  x,y;\cdot\right)  $ is convex and
differentiable, then under natural moment assumptions which guarantee that
expectations are well defined, we have that for
\[
P\in\mathcal{U}_{\delta}\left(  P_{n}\right)  =\{P:D_{c}\left(  P,P_{n}%
\right)  \leq\delta\},
\]
the parameter $\beta^{P}$ must satisfy
\begin{equation}
\mathbb{E}_{P}\left[  \nabla_{\beta}l\left(  X,Y;\beta^{P}\right)  \right]
=0. \label{Est_Eq}%
\end{equation}
Now, for any given $\beta$, let us define
\[
\mathcal{M}\left(  \beta\right)  =\left\{  P:\mathbb{E}_{P}\left[
\nabla_{\beta}l\left(  X,Y;\beta\right)  \right]  =0\right\}  ,
\]
which is the set of probability measures $P$, under which $\beta$ is the
optimal risk minimization parameter. We would like to choose $\delta$ as small
as possible so that
\begin{equation}
\mathcal{U}_{\delta}\left(  P_{n}\right)  \cap\mathcal{M}\left(  \beta^{\ast
}\right)  \neq\varnothing\label{Int_Non_Empty}%
\end{equation}
with probability at least $1-\chi$. But note that (\ref{Int_Non_Empty}) holds
if and only if there exists $P$ such that $D_{c}\left(  P,P_{n}\right)
\leq\delta$ and $\mathbb{E}_{P}\left[  \nabla_{\beta}l\left(  X,Y;\beta^{\ast
}\right)  \right]  =0$. \newline The RWP function is defined
\begin{equation}
R_{n}\left(  \beta\right)  =\min\{D_{c}\left(  P,P_{n}\right)  :\mathbb{E}%
_{P}\left[  \nabla_{\beta}l\left(  X,Y;\beta\right)  \right]  =0\}.
\label{RWP_Function}%
\end{equation}
In view of our discussion following (\ref{Int_Non_Empty}), it is immediate
that $\beta^{\ast}\in\Lambda_{\delta}\left(  P_{n}\right)  $ if and only if
$R_{n}\left(  \beta^{\ast}\right)  \leq\delta$, which then implies that
\[
\delta_{n}^{\ast}=\inf\{\delta:P\left(  R_{n}\left(  \beta^{\ast}\right)
\leq\delta\right)  >1-\chi\}.
\]
Consequently, we conclude that $\delta_{n}^{\ast}$ can be evaluated
asymptotically in terms of the $1-\chi$ quantile of $R_{n}\left(  \beta^{\ast
}\right)  $ and therefore we must identify the asymptotic distribution of
$R_{n}\left(  \beta^{\ast}\right)  $ as $n\rightarrow\infty$. We illustrate
intuitively the role of the RWP function and $\mathcal{M}\left(  \beta\right)
$ in Figure 1, where RWP function $R_{n}\left(  \beta^{\ast}\right)  $ could
be interpreted as the discrepancy distance between empirical measure $P_{n}$
and the manifold $\mathcal{M}\left(  \beta^{\ast}\right)  $ associated with
$\beta^{\ast}$. \begin{figure}[pth]
\begin{center}
\includegraphics[width=0.8\textwidth]{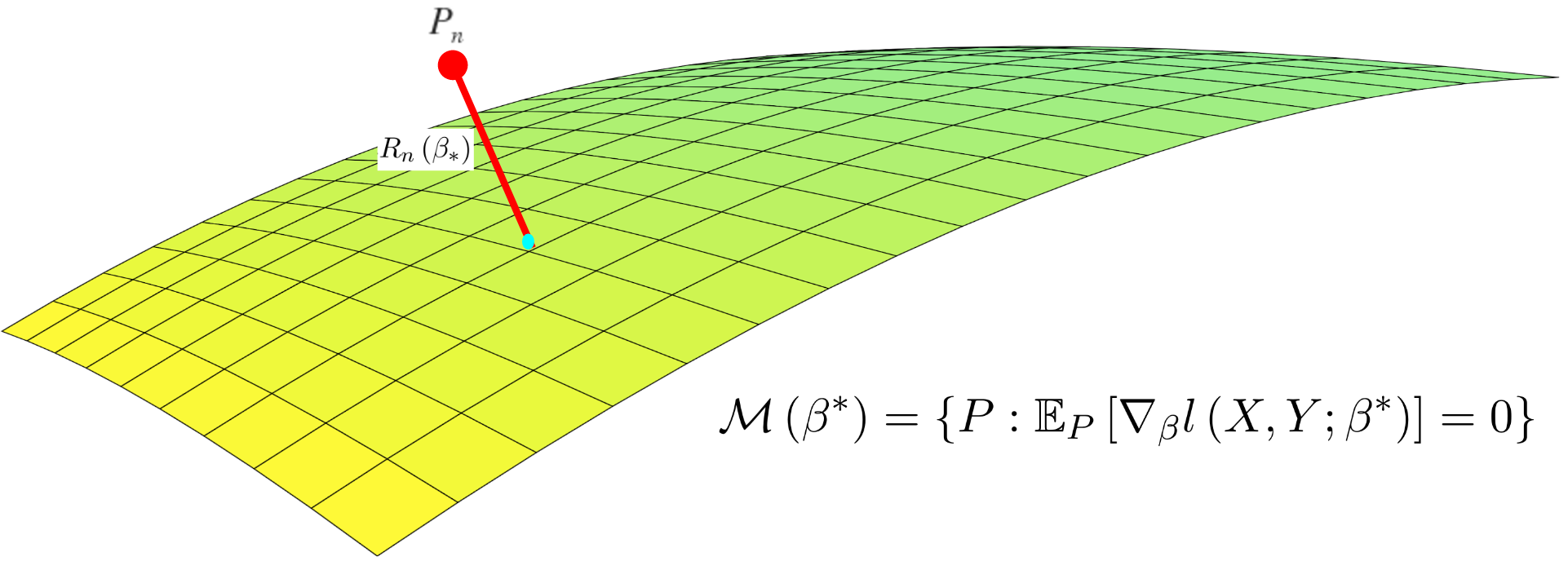}
\end{center}
\caption{Intuitive Plot for the RWP function $R_{n}\left(  \beta\right)  $ and
the set $\mathcal{M}\left(  \beta\right)  $.}%
\label{fig:spiral}%
\end{figure}\newline Typically, under assumptions supporting the underlying
model (as in the generalized linear setting we considered), we will have that
$\beta^{\ast}$ is characterized by the estimating equation (\ref{Est_Eq}).
Therefore, under natural statistical assumptions one should expect that
$R_{n}\left(  \beta^{\ast}\right)  \rightarrow0$ as $n\rightarrow\infty$ at a
certain rate and therefore $\delta_{n}^{\ast}\rightarrow0$ at a certain
(optimal) rate. This then yields an optimal rate of convergence to zero for
the underlying regularization parameter. The next subsections will investigate
the precise rate of convergence analysis of $\delta_{n}^{\ast}$.

\subsection{Optimal Regularization for GSRL\ Linear
Regression\label{Sect_3_3_RWP_Lin}}

We assume, for simplicity, that the training data set $\{(X_{1},Y_{1}%
),\ldots,(X_{n},Y_{n})\}$ is i.i.d. and that the linear relationship
$Y_{i}=\beta^{\ast\;T}X_{i}+e_{i}$, holds with the errors $\left\{
e_{1},...,e_{n}\right\}  $ being i.i.d. and independent of $\{X_{1}%
,\ldots,X_{n}\}$. Moreover, we assume that both the entries of $X_{i}$ and the
errors have finite second moment and the errors have zero mean. \newline
%Since in our current setting with square loss function the gradient of square loss is, $\nabla_{\beta}\left(Y-X^{T}\beta\right)^{2} = - X(Y-X^{T}\beta)$, then the RWP function (\ref{RWP_Function})\ for linear
%regression model is given as,
Since in our current setting $l\left(  x,y;\beta\right)  =\left(  y-x^{T}%
\beta\right)  ^{2}$, then the RWP function (\ref{RWP_Function})\ for linear
regression model is given as,
\begin{equation}
R_{n}\left(  \beta\right)  =\min_{P}\left\{  D_{c}\left(  P,P_{n}\right)
:\mathbb{E}_{P}\left[  X\left(  Y-X^{T}\beta\right)  \right]  =0\right\}  .
\label{Eqn-RWP-Linear}%
\end{equation}

\begin{theorem}
[RWP Function Asymptotic Results: Linear Regression]%
\label{Thm-RWPI-Asymptotic_Lin}Under the assumptions imposed in this
subsection and the cost function as given in \eqref{Def_c_LR}, with
$\varrho=2$,
\[
nR_{n}\left(  \beta^{\ast}\right)  \Rightarrow L_{1}:=\max_{\zeta\in
\mathbb{R}^{d}}\left\{  2\sigma\zeta^{T}Z-\mathbb{E}\left[  \left\Vert
e\zeta-\left(  \zeta^{T}X\right)  \beta^{\ast}\right\Vert _{\alpha
\text{-}(p,s)}^{2}\right]  \right\}  ,
\]
as $n\rightarrow\infty$, where $\Rightarrow$ means convergence in distribution
and $Z\sim\mathcal{N}\left(  0,\Sigma\right)  $ with $\Sigma=Var(X)$.
Moreover, we can observe the more tractable stochastic upper bound,
\[
L_{1}\overset{D}{\leq}L_{2}:=\frac{\mathbb{E}\left[  e^{2}\right]
}{\mathbb{E}\left[  e^{2}\right]  -\left(  \mathbb{E}\left[  \left\vert
e\right\vert \right]  \right)  ^{2}}\left\Vert Z\right\Vert _{\alpha
^{-1}\text{-}(q,t)}^{2}.
\]

\end{theorem}

\bigskip

\textbf{We now explain how to use Theorem \ref{Thm-RWPI-Asymptotic_Lin} to set
the regularization parameter in GSRL linear regression}:

\begin{enumerate}
\setlength\itemsep{.3em}

\item Estimate the $1-\chi$ quantile of $\left\Vert Z\right\Vert _{\alpha
^{-1}\text{-}(q,t)}^{2}$. We use use $\hat{\eta}_{1-\chi}$ to denote the
estimator for this quantile. This step involves estimating $\Sigma$ from the
training data.

\item The regularization parameter $\lambda$ in the GSRL linear regression
takes the form
\[
\lambda=\sqrt{\delta}=\hat{\eta}_{1-\chi}^{1/2}\left(  n(1-\left(
\mathbb{E}\left\vert e\right\vert \right)  ^{2}/\mathbb{E}e^{2})\right)
^{-1/2}.
\]
Note that the denominator in the previous expression must be estimated from
the training data.
\end{enumerate}
\bigskip

Note that the regularization parameter for GSRL for linear regression chosen
via our RWPI asymptotic result does not depends on the magnitude of error $e$
(see also the discussion in \citeapos{bunea2014group}). \newline It is also possible to formulate the
optimal regularization results for high-dimension setting, where the number of
predictors growth with sample size. We discuss 
the results in the Appendix namely Section \ref{Sec-Supply-Highdim}.

\subsection{Optimal Regularization for GR-Lasso\ Logistic
Regression\label{Sect_3_4_Log}}

We assume that the training data set $\{(X_{1},Y_{1}),\ldots,(X_{n},Y_{n})\}$
is i.i.d.. In addition, we assume that the $X_{i}$'s have a finite second
moment and also that they possess a density with respect to the Lebesgue
measure. Moreover, we assume a logistic regression model; namely,
\begin{equation}
P\left(  Y_{i}=1|X_{i}\right)  =1/\left(  1+\exp\left(  -X_{i}^{T}\beta^{\ast
}\right)  \right)  , \label{Eqn-Logistic-Model}%
\end{equation}
and $P\left(  Y_{i}=-1|X_{i}\right)  =1-P\left(  Y_{i}=1|X_{i}\right)  $.
\newline In the logistic regression setting, we consider the log-exponential
loss defined in \eqref{E_log_Expo_loss}. Therefore, the RWP function,
(\ref{RWP_Function}),\ for logistic regression is
\begin{equation}
R_{n}\left(  \beta\right)  =\min\left\{  D_{c}\left(  P,P_{n}\right)
:\mathbb{E}_{P}\left[  \frac{YX}{1+\exp\left(  YX^{T}\beta\right)  }\right]
=0\right\}  . \label{Eqn-RWP-Logistic}%
\end{equation}

\begin{theorem}
[RWP Function Asymptotic Results: Logistic Regression]%
\label{Thm-RWPI-Asymptotic} Under the assumptions imposed in this subsection
and the cost function as given in \eqref{Def_c_LR} with $\varrho=1$,
\[
\sqrt{n}R_{n}\left(  \beta^{\ast}\right)  \Rightarrow L_{3}:=\sup_{\zeta\in
A}\quad\zeta^{T}Z,
\]
as $n\rightarrow\infty$, where
\[
Z\sim\mathcal{N}\left(  0,\mathbb{E}\left[  \frac{XX^{T}}{\left(
1+\exp\left(  YX^{T}\beta^{\ast}\right)  \right)  ^{2}}\right]  \right)
\]
and
\[
A=\left\{  \zeta\in\mathbb{R}^{d}:\mathrm{ess}\sup_{X,Y}\left\Vert \zeta
^{T}\frac{y\left(  1+\exp\left(  YX^{T}\beta^{\ast}\right)  \right)
I_{d\times d}-XX^{T}}{\left(  1+\exp\left(  YX^{T}\beta^{\ast}\right)
\right)  ^{2}}\right\Vert _{\alpha\text{-}(p,s)}\leq1\right\}  .
\]
Further, the limit law $L_{3}$ follows the simpler stochastic bound,
\[
L_{3}\overset{D}{\leq}L_{4}:=\left\Vert \tilde{Z}\right\Vert _{\alpha
^{-1}\text{-}(q,t)},
\]
where $\tilde{Z}\sim\mathcal{N}\left(  0,\Sigma\right)  $.
\end{theorem}

\textbf{We now explain how to use Theorem \ref{Thm-RWPI-Asymptotic} to set the
regularization parameter in GR-Lasso logistic regression.}

\begin{enumerate}
\setlength\itemsep{.5em}

\item Estimate the $1-\chi$ quantile of $L_{4}$. We use use $\hat{\eta
}_{1-\chi}$ to denote the estimator for this quantile. This step involves
estimating $\Sigma$ from the training data.

\item We choose the regularization parameter $\lambda$ in the GR-Lasso problem
to be,
\[
\lambda=\delta=\hat{\eta}_{1-\chi}/\sqrt{n}.
\]

\end{enumerate}

\section{Numerical Experiments\label{Sec-Numerical}}

We proceed to numerical experiments on both simulated and real data to verify
the performance of our method for choosing the regularization parameter. We
apply \textquotedblleft grpreg\textquotedblright\ in R, from
\citeapos{breheny2016package}, to solve GR-Lasso for logistic regression. For GSRL
for linear regression, we consider apply the \textquotedblleft
grpreg\textquotedblright\ solver for the GR-Lasso problem combined with the
iterative procedure discussed in Section 2 of \citeapos{sun2011scaled} (see also
Section 5 of \citeapos{li2015flare} for the Lasso counterpart of such numerical
procedure). 
\bigskip\newline
\textbf{Data preparation for simulated experiments:} We
borrow the setting from example III in \citeapos{yuan2006model}, where the group
structure is determined by the third order polynomial expansion. More
specifically, we assume that we have 17 random variables $Z_{1},\dots,Z_{16}$
and $W$, they are i.i.d. and follow the normal distribution. The
covariates $X_{1},\dots,X_{16}$ are given as $X_{i}=\left(  Z_{i}+W\right)
/\sqrt{2}$. For the predictors, we consider each covariate and its second and
third order polynomial, i.e. $X_{i}$, $X_{i}^{2}$ and $X_{i}^{3}$. In total,
we have $48$ predictors. \newline\textbf{For linear regression}: The response
$Y$ is given by
\[
Y=\beta_{3,1}X_{3}+\beta_{3,2}X_{3}^{2}+\beta_{3,3}X_{3}^{3}+\beta_{5,1}%
X_{5}+\beta_{5,2}X_{5}^{2}+\beta_{5,3}X_{5}^{3}+e,
\]
where $\beta_{(\cdot,\cdot)}$ coefficients draw randomly and $e$ represents an independent random error.
\newline\textbf{For classification:} We consider $Y$ simulated by a Bernoulli
distribution, i.e.
\[
Y\sim Ber\left(  {1}/{\left[  1+\exp\left(  -\left(  \beta_{3,1}X_{3}%
+\beta_{3,2}X_{3}^{2}+\beta_{3,3}X_{3}^{3}+\beta_{5,1}X_{5}+\beta_{5,2}%
X_{5}^{2}+\beta_{5,3}X_{5}^{3}\right)  \right)  \right]  }\right)  .
\]
We compare the following methods for linear regression and logistic
regression: 1) groupwise regularization with asymptotic results (in Theorem
\ref{Thm-RWPI-Asymptotic_Lin}, \ref{Thm-RWPI-Asymptotic}) selected tuning
parameter (RWPI GRSL and RWPI GR-Lasso), 2) groupwise regularization with
cross-validation (CV GRSL and CV GR-Lasso), and 3) ordinary least square and
logistic regression (OLS and LR).

We report the error as the square loss for linear regression and
log-exponential loss for logistic regression. The training error is calculated
via the training data. The size of the training data is taken to be
$n=50,100,500$ and $1000$. The testing error is evaluated using a simulated
data set of size $1000$ using the same data generating process described
earlier. The mean and
standard deviation of the error are reported via $200$ independent runs of the
whole experiment, for each sample size $n$.

The detailed results are summarized in Table \ref{Table-Linear-GLasso} for
linear regression and Table \ref{Table-Logistic-GLasso} for logistic
regression. We can see that our procedure is very comparable to cross
validation, but it is significantly less time consuming and all of the data
can be directly used to estimate the model parameter, by-passing significant
data usage in the estimation of the regularization parameter via cross
validation\begin{table}[th]
{\small \centering
}
\par
{\small
\begin{tabular}
[c]{c|cc|cc|cc}
& \multicolumn{2}{c|}{RWPI GSRL} & \multicolumn{2}{c|}{CV GSRL} &
\multicolumn{2}{c}{OLS}\\\hline
Sample Size & Training & Testing & Training & Testing & Training & Testing\\
$n=50$ & $5.64\pm1.16$ & $9.15\pm3.58$ & $3.18 \pm1.07$ & $7.66 \pm2.69$ &
$0.07 \pm0.09$ & $80.98\pm30.53$\\
$n=100$ & $4.67\pm0.70$ & $5.83\pm1.38$ & $3.61\pm0.74$ & $5.22\pm1.05$ &
$2.09 \pm0.44$ & $73.35\pm16.51$\\
$n=500$ & $4.09\pm0.29$ & $4.16\pm0.27$ & $3.93\pm0.3$ & $4.12\pm0.27$ &
$3.63\pm0.27$ & $73.08\pm10.40$\\
$n=1000$ & $4.02\pm0.19$ & $4.11\pm0.26$ & $3.95\pm0.19$ & $4.11\pm0.26$ &
$3.82\pm0.19$ & $72.28\pm8.05$%
\end{tabular}
}\caption{Linear Regression Simulation Results. }%
\label{Table-Linear-GLasso}%
\end{table}\begin{table}[th]
{\small \centering
}
\par
{\small
\begin{tabular}
[c]{c|cc|cc|cc}
& \multicolumn{2}{c|}{RWPI GR-Lasso} & \multicolumn{2}{c|}{CV GR-Lasso} &
\multicolumn{2}{c}{Logistic Regression}\\\hline
Sample Size & Training & Testing & Training & Testing & Training & Testing\\
$n=50$ & $.683 \pm.016$ & $.702 \pm.014$ & $.459\pm.118$ & $.628\pm.099$ &
$.002\pm.001$ & $5.288 \pm1.741$\\
$n=100$ & $.593 \pm.038$ & $.618 \pm.029$ & $.450 \pm.061$ & $.551 \pm.037$ &
$.042 \pm.041$ & $4.571 \pm1.546$\\
$n=500$ & $.513 \pm.021$ & $.518 \pm.019$ & $.461 \pm.025$ & $.493\pm.018$ &
$.083 \pm.057$ & $1.553 \pm.355$\\
$n=1000$ & $.492 \pm.016$ & $.488 \pm.017$ & $.491 \pm.017$ & $.488 \pm.019$ &
$.442\pm.018$ & $.510 \pm.028$%
\end{tabular}
}\caption{Logistic Regression Simulation Results. }%
\label{Table-Logistic-GLasso}%
\end{table}
\bigskip\newline 
We also validated our method using the Breast Cancer
classification problem with data from the UCI machine learning database
discussed in \citeapos{Lichman:2013}. The data set contains $569$ samples with one
binary response and $30$ predictors. We consider all the predictors and their
first, second, and third order polynomial expansion. Thus, we end up having
$90$ predictors divided into $30$ groups. For each iteration, we randomly
split the data into a training set with $112$ samples and the rest in the
testing set. We repeat the experiment $500$ times to observe the
log-exponential loss function for the training and testing error. We compare
our asymptotic results based GR-Lasso logistic regression (RWPI GR-Lasso),
cross-validation based GR-Lasso logistic regression (CV GR-Lasso), vanilla
logistic regression (LR), and regularized logistic regression (LRL1). We can
observe, even when the sample size is small as in the example, our method
still provides very comparable results (see in Table \ref{LogReg_table-real}).
\begin{table}[th]
{\footnotesize \centering
\begin{tabular}
[c]{cc|cc|cc|cc}%
\multicolumn{2}{c|}{LR} & \multicolumn{2}{c|}{LRL1} & \multicolumn{2}{c|}{RWPI
GR-Lasso} & \multicolumn{2}{c}{CV GR-Lasso}\\\hline
Training & Testing & Training & Testing & Training & Testing & Training &
Testing\\
$0.0\pm0.0$ & $15.267\pm5.367$ & $.510\pm.215$ & $.414\pm.173$ & $.186\pm.032$
& $.240\pm.098$ & $.198\pm.041$ & $.213\pm.041$%
\end{tabular}
}\caption{Numerical results for breast cancer data set. }%
\label{LogReg_table-real}%
\end{table}

\section{Conclusion and Extensions\label{Sect_Conclusion}}

Our discussion of GSRL as a DRO problem has exposed rich interpretations which
we have used to understand GSRL's generalization properties by means of a game
theoretic formulation. Moreover, our DRO representation also elucidates the
crucial role of the regularization parameter in measuring the distributional
uncertainty present in the data. Finally, we obtained asymptotically valid
formulas for optimal regularization parameters under a criterion which is
naturally motivated, once again, thanks to our DRO\ formulation. Our
easy-to-implement formulas are shown to perform well compared to
(time-consuming) cross validation. 
\bigskip\newline We strongly believe that
our discussion in this paper can be easily extended to a wide range of machine
learning estimators. We envision formulating the DRO problem considering
different types of models and cost functions. We plan to investigate
algorithms which solve the DRO problem directly (even if no direct
regularization representation, as the one we considered here, exists).
Moreover, it is natural to consider different types of cost functions which
might improve upon the simple choice which, as we have shown, implies the GSRL
estimator. Questions related to alternative choices of cost functions are also
under current research investigations, and our progress will be reported in
the near future.
%In the statement of the previous result, the symbol $W_{n}\lesssim_{D}W$
%means
%\[
%\overline{\lim}_{n\rightarrow\infty}P\left(  W_{n}>x\right)  \leq P\left(
%W>x\right)
%\]
%for any $x$. The previous result

%\section{Section Title}
%Main contents here.
%\subsection{Subsection Title}
%A figure in Fig.~\ref{fig:spiral}. Please use high quality graphics for your camera-ready submission -- if you can use a vector graphics format such as \texttt{.eps} or \texttt{.pdf}.
%\begin{figure}[htp]
%\begin{center}
%\includegraphics[width=0.5\textwidth]{spiral.eps}
%\caption{A spiral.}\label{fig:spiral}
%\end{center}
%\end{figure}
%An example of citation~\citeapos{DBLP:conf/acml/2009}.

%\acks{Acknowledgements should go at the end, before appendices and references.}

%\bibliographystyle{plain}
\bibliographystyle{apalike}
\bibliography{DRO_GroupLasso}

\appendix

\section{Technical Proofs\label{Sect_Appendix}}

We will first derive some properties for $\alpha$-$(p,s)$ norm we defined in
\eqref{alpha_ps_norm}, then we move to the proof for DRO problem in Section
\ref{Sect_Appendix_DRO} and the optimal selection of regularization parameter
in Section \ref{Sect_Appendix_RWP}. We will focus on
the proof for linear regression and leave the part for logistic regression,
which follows the similar techniques, in the Appendix B.

\subsection{Basic Properties of the $\alpha$-$(p,s)$ Norm
\label{Sec_Appendix_alpha-norm}}

The following Proposition, which describes basic properties of the $\alpha
$-$(p,s)$ norm, will be very useful in our proofs.

\begin{proposition}
\label{Thm-Dual-Norm} For $\alpha-(p,s)$ norm defined for $\mathbb{R}^{d}$ as
in \eqref{alpha_ps_norm} and the notations therein, we have the following
properties: \newline\textbf{I)} The dual norm of $\alpha-(p,s)$ norm is
$\alpha^{-1}\text{-}(q,t)$ norm, where $\alpha^{-1}=\left(  1/\alpha
_{1},\ldots,1/\alpha_{\bar{d}}\right)  ^{T}$, $1/p+1/q=1$, and $1/s+1/t=1$
(i.e. $p,q$ are conjugate and $s,t$ are conjugate). \newline\textbf{II)} The
H\"{o}lder inequality holds for the $\alpha\text{-}(p,s)$ norm, i.e. for
$a,b\in\mathbb{R}^{d}$, we have,
\[
a^{T}b\leq\left\Vert a\right\Vert _{\alpha\text{-}(p,s)}\left\Vert
b\right\Vert _{\alpha^{-1}\text{-}(q,t)},
\]
where the equality holds if and only if $sign(a(G_{j})_{i})=sign(b(G_{j}%
)_{i})$ and
\[
\left\vert \alpha_{j}a(G_{j})_{i}\right\vert \left\Vert \frac{1}{\alpha_{j}%
}b(G_{j})\right\Vert _{q}^{q/p-t/s}\left\Vert b\right\Vert _{\alpha
^{-1}\text{-}(q,t)}^{t/s}={\left\vert \frac{1}{\alpha_{j}}b(G_{j}%
)_{i}\right\vert ^{q/p}}.
\]
%\[
%	\left\Vert b\right\Vert_{\alpha^{-1}\text{-}(q,t)}
%	\left\Vert \alpha_{j} a(G_j)\right\Vert_{p}^{s/t}
%	\left\vert \alpha_{j} a(G_j)_i\right\vert^{p/q} =
%	\left\Vert a\right\Vert_{\alpha\text{-}(p,s)}^{s/t}
%	\left\Vert \alpha_{j} a(G_j)\right\Vert_{p}^{p/q}
%		\left\vert \frac{1}{\alpha_{j}} b(G_j)_i\right\vert,
%\]
is true for all $j=1,\ldots,\bar{d}$ and $i=1,\ldots,g_{j}$. \newline The
triangle inequality holds, i.e. for $a,b\in\mathbb{R}^{d}$ and $a\neq0$, we
have
\[
\left\Vert a\right\Vert _{\alpha\text{-}(p,s)}+\left\Vert b\right\Vert
_{\alpha\text{-}(p,s)}\geq\left\Vert a+b\right\Vert _{\alpha\text{-}(p,s)},
\]
where the equality holds if and only if, there exists nonnegative $\tau$, such
that $\tau a=b$.
\end{proposition}

\begin{proof}
[Proof of Proposition \ref{Thm-Dual-Norm}]We first proceed to prove II). Let
us consider any $a,b\in\mathbb{R}^{d}$. We can assume $a,b\neq0$, otherwise
the claims are immediate. The inner product (or dot product) of $a$ and $b$ an
be written as:
\[
a^{T}b=\sum_{j=1}^{\bar{d}}\left[  \sum_{i=1}^{g_{j}}a(G_{j})_{i}b(G_{j}%
)_{i}\right]  \leq\sum_{j=1}^{\bar{d}}\left[  \sum_{i=1}^{g_{j}}\left\vert
a(G_{j})_{i}\right\vert \cdot\left\vert b(G_{j})_{i}\right\vert \right]  .
\]
The equality holds for the above inequality if and only if $a(G_{j})_{i}$ and
$b(G_{j})_{i}$ shares the same sign. For each fixed $j=1,\ldots,\bar{d}$, we
consider the term in the bracket,
\[
\sum_{i=1}^{g_{j}}\left\vert a(G_{j})_{i}\right\vert \cdot\left\vert
b(G_{j})_{i}\right\vert =\sum_{i=1}^{g_{j}}\alpha_{j}\left\vert a(G_{j}%
)_{i}\right\vert \cdot\left\vert b(G_{j})_{i}\right\vert /\alpha_{j}%
\leq\left\Vert \alpha_{j}a(G_{j})\right\Vert _{p}\cdot\left\Vert \frac
{1}{\alpha_{j}.}b(G_{j})\right\Vert _{q}.
\]
The above inequality is due to H\"{o}lder's inequality for $p-$norm and the
equality holds if and only if
\[
\left\Vert \frac{1}{\alpha_{j}.}b(G_{j})\right\Vert _{q}^{q}\left\vert
\alpha_{j}a(G_{j})_{i}\right\vert ^{p}=\left\Vert \alpha_{j}a(G_{j}%
)\right\Vert _{p}^{p}\left\vert \frac{1}{\alpha_{j}}b(G_{j})_{i}\right\vert
^{q},
\]
is true for all $i=\overline{1,g_{j}}$. Combining the above result for each
$j=1,\ldots,\bar{d}$, we have,
\[
a^{T}b\leq\sum_{j=1}^{\bar{d}}\left\Vert \alpha_{j}a(G_{j})\right\Vert
_{p}\cdot\left\Vert \frac{1}{\alpha_{j}}b(G_{j})\right\Vert _{q}\leq\left\Vert
a\right\Vert _{\alpha\text{-}(p,s)}\cdot\left\Vert b\right\Vert _{\alpha
^{-1}\text{-}(q,t)},
\]
where the final inequality is due to H\"{o}lder inequality applied to the
vectors
\begin{align}
\tilde{a}  &  =\left(  \alpha_{1}\left\Vert a(G_{1})\right\Vert _{p}%
,\ldots,\alpha_{\bar{d}}\left\Vert a(G_{\bar{d}})\right\Vert _{p}\right)
^{T},\text{ }\label{Eqn-group-vec}
\text{and }\tilde{b}    =\left(  \frac{1}{\alpha_{1}}\left\Vert b_{G_{1}%
}\right\Vert _{q},\ldots,\frac{1}{\alpha_{\bar{d}}}\left\Vert b(G_{\bar{d}%
})\right\Vert _{q}\right)  ^{T}.
\end{align}
This proves the H\"{o}lder type inequality stated in the theorem. We can
further observe that the final inequality becomes equality if and only if
\[
\left\Vert b\right\Vert _{\alpha^{-1}-(q,t)}^{t}\left\Vert \alpha_{j}%
a(G_{j})\right\Vert _{p}^{s}=\left\Vert a\right\Vert _{\alpha-(p,s)}%
^{s}\left\Vert \frac{1}{\alpha_{j}}b(G_{j})\right\Vert _{q}^{t},
\]
holds for all $j=1,\ldots,\bar{d}$. Combining the conditions for equalities
hold for each inequality, we conclude condition II) in the statement of the proposition.
\newline Now we proceed to prove I). Recall the definition of a dual norm,
i.e. $$\left\Vert b\right\Vert _{\alpha\text{-}(p,s)}^{\ast}=\sup_{a:\left\Vert
a\right\Vert _{\alpha\text{-}(p,s)}=1}a^{T}b$$. Now, choose $b\in\mathbb{R}%
^{d}$, $b\neq0$, and let us take $a$ satisfying, $\left\Vert a\right\Vert
_{\alpha-(p,s)}=1$ and
\[
a(G_{j})_{i}=\frac{sign(b(G_{j})_{i})}{\alpha_{j}}\frac{\left\vert \frac
{1}{\alpha_{j}}b(G_{j})_{i}\right\vert ^{q/p}}{\left\Vert \frac{1}{\alpha_{j}%
}b(G_{j})\right\Vert _{q}^{q/p-t/s}\left\Vert b\right\Vert _{\alpha
^{-1}\text{-}(q,t)}^{t/s}}.
\]
By part II), we have that
\[
\left\Vert b\right\Vert _{\alpha-(p,s)}^{\ast}=\sup_{a:\left\Vert a\right\Vert
_{\alpha-(p,s)}=1}a^{T}b=\left\Vert a\right\Vert _{\alpha-(p,s)}\left\Vert
b\right\Vert _{\alpha^{-1}-(q,t)}=\left\Vert b\right\Vert _{\alpha^{-1}%
-(q,t)}.
\]
Thus we proved part I). Finally, let us discuss the triangle inequality. For
any $a,b\in\mathbb{R}^{d}$ and $a,b\neq0$ we have
\begin{align*}
&  \quad\quad\left\Vert a\right\Vert _{\alpha\text{-}(p,s)}+\left\Vert
b\right\Vert _{\alpha\text{-}(p,s)}\\
&=\left[  \sum_{j=1}^{\bar{d}}\alpha
_{j}\left\Vert a(G_{j})\right\Vert _{p}^{s}\right]  ^{1/s}+\left[  \sum
_{j=1}^{\bar{d}}\alpha_{j}\left\Vert b(G_{j})\right\Vert _{p}^{s}\right]
^{1/s}\\
&  \geq\left[  \sum_{j=1}^{\bar{d}}\alpha_{j}\left(  \left\Vert a(G_{j}%
)\right\Vert _{p}^{s}+\left\Vert b(G_{j})\right\Vert _{p}^{s}\right)  \right]
^{1/s}\\
&\geq\left[  \sum_{j=1}^{\bar{d}}\alpha_{j}\left\Vert a(G_{j}%
)+b(G_{j})\right\Vert _{p}^{s}\right]  ^{1/s}\\
&  =\left\Vert a+b\right\Vert _{\alpha\text{-}(p,s)}.
\end{align*}
For the above derivation, the first equality is due to definition in
\eqref{alpha_ps_norm}, Second equality is applying the triangle inequality of
$s$-norm for $\tilde{a}$ and $\tilde{b}$ defined in \eqref{Eqn-group-vec},
where the equality holds if and only if, there exist positive number
$\tilde{\tau}$, such that $\tilde{\tau}\tilde{a}=\tilde{b}$. Third inequality
is due to triangle equality of $p$-norm to $a(G_{j})$ and $b(G_{j})$ for each
$j=1,\ldots,\bar{d}$, where the equality holds if and only if, there exists
nonnegative numbers $\tau_{j}$, such that $\tau_{j}a(G_{j})=b(G_{j})$.
Combining the equality condition for second and third estimate above, we can
conclude the equality condition for the triangle inequality for $\alpha
\text{-}(p,s)$ norm is if and only if there exists a non-negative number
$\tau$, such that $\tau a=b$.
\end{proof}

\subsection{Proof of DRO for Linear Regression \label{Sect_Appendix_DRO}}

\begin{proof}
[Proof of Theorem \ref{Thm-Dro-GLasso-Regression}]Let us apply the strong
duality results, as in the Appendix of \citeapos{blanchet2016robust}, for worst-case expected loss function, which
is a semi-infinity linear programming problem, and write the worst-case loss
as,
\[
\sup_{P:D_{c}\left(  P,P_{n}\right)  \leq\delta}\mathbb{E}_{P}\left[  \left(
Y-X^{T}\beta\right)  ^{2}\right]  =\min_{\gamma\geq0}\left\{  \gamma
\delta-\frac{1}{n}\sum_{i=1}^{n}\sup_{u}\left\{  \left(  y_{i}-u^{T}%
\beta\right)  ^{2}-\gamma\left\Vert x_{i}-u\right\Vert _{\alpha^{-1}%
\text{-}(q,t)}^{2}\right\}  \right\}  .
\]
For each $i$, let us consider the inner optimization problem over $u$. We can
denote $\Delta=u-x_{i}$ and $e_{i}=y_{i}-x_{i}^{T}\beta$ for notation
simplicity, then the $i-$th inner optimization problem becomes,
\begin{align}
&  e_{i}^{2}+\sup_{\Delta}\left\{  \left(  \Delta^{T}\beta\right)  ^{2}%
-2e_{i}\Delta^{T}\beta-\gamma\left\Vert \Delta\right\Vert _{\alpha^{-1}%
-(q,t)}^{2}\right\}  \nonumber\\
= &  e_{i}^{2}+\sup_{\Delta}\left\{  \left(  \sum_{j}\left\vert \Delta
_{j}\right\vert \left\vert \beta_{j}\right\vert \right)  ^{2}+2\left\vert
e_{i}\right\vert \sum_{j}\left\vert \Delta_{j}\right\vert \left\vert \beta
_{j}\right\vert -\gamma\left\Vert \Delta\right\Vert _{\alpha^{-1}-(q,t)}%
^{2}\right\}  \nonumber\\
= &  e_{i}^{2}\sup_{\left\Vert \Delta\right\Vert _{\alpha^{-1}\text{-}(q,t)}%
}\left\{  \left\Vert \beta\right\Vert _{\alpha\text{-}(p,s)}^{2}\left\Vert
\Delta\right\Vert _{\alpha^{-1}\text{-}(q,t)}^{2}+2\left\Vert \beta\right\Vert
_{\alpha\text{-}(p,s)}\left\vert e_{i}\right\vert \left\Vert \Delta\right\Vert
_{\alpha^{-1}\text{-}(q,t)}-\gamma\left\Vert \Delta\right\Vert _{\alpha
^{-1}\text{-}(q,t)}^{2}\right\}  \nonumber\\
= &  \left\{
\begin{array}
[c]{rcl}%
e_{i}^{2}\frac{\gamma}{\gamma-\left\Vert \beta\right\Vert _{\alpha
\text{-}(p,s)}^{2}} & \text{ if }\gamma>\left\Vert \beta\right\Vert
_{\alpha\text{-}(p,s)}^{2}, & \\
+\infty\text{ } & \text{ if }\gamma\leq\left\Vert \beta\right\Vert
_{\alpha\text{-}(p,s)}^{2}. &
\end{array}
\right.  ,\label{14a}%
\end{align}
where the development uses the duality results developed in Proposition
\ref{Thm-Dual-Norm}. The last equality is optimize over $\Delta$ for two
different cases of $\lambda$. \newline Since optimization over $\gamma$ is a
minimization, the outer player will always select $\gamma$ that avoids an
infinite value of the game. Then we can write the worst-case expected loss
function as,
\begin{align}
&\sup_{P:D_{c}\left(  P,P_{n}\right)  \leq\delta}\mathbb{E}_{P}\left[  \left(
Y-X^{T}\beta\right)  ^{2}\right]\label{14b} \\
 = &  \min_{\gamma>\left\Vert \beta
\right\Vert _{\alpha\text{-}(p,s)}^{2}}\left\{  \gamma\delta-\gamma
\frac{E_{P_{n}}l\left(  X,Y;\beta\right)  }{\gamma-\left\Vert \beta\right\Vert
_{\alpha\text{-}(p,s)}^{2}}\right\} \nonumber \\
= &  \left(  \sqrt{E_{P_{n}}l\left(  X,Y;\beta\right)  }+\sqrt{\delta
}\left\Vert \beta\right\Vert _{\alpha\text{-}(p,s)}\right)  ^{2}.\nonumber
\end{align}
The first equality in (\ref{14b}) is a plug-in from the result in (\ref{14a}).
For the second equality, we can observe the target function is convex and
differentiable and as $\gamma\rightarrow\infty$ and $\gamma\rightarrow
\left\Vert \beta\right\Vert _{\alpha\text{-}(p,s)}^{2}$, the value function
will be infinity. We can solve this convex optimization problem which leads to
the result above. We further take square root and take minimization over
$\beta$ on both sides, we proved the claim of the theorem.
\end{proof}

\subsection{Proof for Optimal Selection of Regularization for Linear
Regression\label{Sect_Appendix_RWP}}

\begin{proof}
[Proof for Theorem \ref{Thm-RWPI-Asymptotic_Lin}]For linear regression with
the square loss function, if we apply the strong duality result for
semi-infinity linear programming problem as in Section B of
\citeapos{blanchet2016robust}, we can write the scaled RWP function for linear
regression as
\begin{equation}
nR_{n}\left(  \beta^{\ast}\right)  =\sup_{\zeta}\left\{  -\zeta^{T}%
Z_{n}-\mathbb{E}_{P_{n}}\phi\left(  X_{i},Y_{i},\beta^{\ast},\zeta\right)
\right\}  ,\label{14c}%
\end{equation}
where $Z_{n}=\frac{1}{\sqrt{n}}\sum_{i=1}^{n}e_{i}X_{i}$ and
\[
\phi\left(  X_{i},Y_{i},\beta^{\ast},\zeta\right)  =\sup_{\Delta}\left\{
e_{i}\zeta^{T}\Delta-\left(  \beta^{\ast\;T}\Delta\right)  \left(  \zeta
^{T}X_{i}\right)  -\left(  \left\Vert \Delta\right\Vert _{\alpha^{-1}%
\text{-}(q,t)}^{2}+n^{-1/2}\left(  \beta^{\ast\; T}\Delta\right)  \left(
\zeta^{T}\Delta\right)  \right)  \right\}  .
\]
Follow the similar discussion in the proof of Theorem 4 in
\citeapos{blanchet2016robust}. Applying Lemma 2 in \citeapos{blanchet2016robust}, we
can argue that the optimizer $\zeta$ can be restrict on a compact set
asymptotically with high probability. We can apply the uniform law of large
number estimate as in Lemma 3 of \citeapos{blanchet2016robust} to the second term
in (\ref{14c}) and we obtain
\begin{equation}
nR_{n}\left(  \beta^{\ast}\right)  =\sup_{\zeta}\{-\zeta^{T}Z_{n}%
-\mathbb{E}\phi\left(  X,Y,\beta,\zeta\right)  ]\}+o_{P}(1).\label{14d}%
\end{equation}
For any fixed $X,Y$, as $n\rightarrow\infty$, we can simplify the contribution
of $\phi\left(  \cdot\right)  $ inside sup in (\ref{14d}). This is done by
applying the duality result (H\"{o}lder-type inequality) in Proposition
\ref{Thm-Dual-Norm} and noting that $\phi\left(  \cdot\right)  $ becomes
quadratic in $\left\Vert \Delta\right\Vert _{\alpha^{-1}\text{-}(q,t)}$. This
results in the simplified expression
\[
nR_{n}\left(  \beta^{\ast}\right)  =\sup_{\zeta}\left\{  -\zeta^{T}%
Z_{n}-\mathbb{E}\left[  \left\Vert e\zeta-(\zeta^{T}X)\beta^{\ast}\right\Vert
_{\alpha\text{-}(p,s)}^{2}\right]  \right\}  +o_{P}(1).
\]
Since we can observe that, $Z_{n}\Rightarrow\sigma Z$, then as $n\rightarrow
\infty$ we proved the first argument. For this step we need to show that the
feasible region can be compactified with high probability. This
compactification argument is done using a technique similar to Lemma 2 in
\citeapos{blanchet2016robust}. \newline By the definition of $L_{1}$, we can apply
H\"{o}lder inequality to the first term, and split the optimization into
optimizing over direction $\left\Vert \zeta^{\prime}\right\Vert _{\alpha
\text{-}(p,s)}=1$ and magnitude $a\geq0$. Thus, we have
\[
L_{1}\leq\max_{\zeta^{\prime}:\left\Vert \zeta^{\prime}\right\Vert
_{\alpha\text{-}(p,s)}=1}\max_{a\geq0}\left\{  2a\sigma\left\Vert Z\right\Vert
_{\alpha^{-1}\text{-}(q,t)}-a^{2}\mathbb{E}\left[  \left\Vert e\zeta^{\prime
}-(\zeta^{\prime T}X)\beta^{\ast}\right\Vert _{\alpha\text{-}(p,s)}%
^{2}\right]  \right\}  .
\]
It is easy to solve the quadratic programming problem in $a$ and we conclude
that
\[
L_{1}\leq\frac{\sigma^{2}\left\Vert Z\right\Vert _{\alpha^{-1}\text{-}%
(q,t)}^{2}}{\min_{\zeta^{\prime}:\left\Vert \zeta^{\prime}\right\Vert
_{\alpha\text{-}(p,s)}=1}\mathbb{E}\left[  \left\Vert e\zeta^{\prime}%
-(\zeta^{\prime T}X)\beta^{\ast}\right\Vert _{\alpha\text{-}(p,s)}^{2}\right]
}.
\]
For the denominator, we have estimates as follows:
\begin{align*}
&  \min_{\zeta^{\prime}:\left\Vert \zeta^{\prime}\right\Vert _{\alpha
\text{-}(p,s)}=1}\mathbb{E}\left[  \left\Vert e\zeta^{\prime}-(\zeta^{\prime
T}X)\beta^{\ast}\right\Vert _{\alpha\text{-}(p,s)}^{2}\right] \\
& \geq\min
_{\zeta^{\prime}:\left\Vert \zeta^{\prime}\right\Vert _{\alpha\text{-}%
(p,s)}=1}\mathbb{E}\left[  \left\vert e\right\vert -\left\vert \zeta
^{T}X\right\vert \left\Vert \beta^{\ast}\right\Vert _{\alpha\text{-}%
(p,s)}\right]  ^{2}\\
&  \geq Var(\left\vert e\right\vert )+\min_{\zeta^{\prime}:\left\Vert
\zeta^{\prime}\right\Vert _{\alpha\text{-}(p,s)}=1}\left(  \left\Vert
\beta^{\ast}\right\Vert _{\alpha\text{-}(p,s)}\mathbb{E}\left\vert
\zeta^{\prime T}X\right\vert -\mathbb{E}\left\vert e\right\vert \right)
^{2}\geq Var(\left\vert e\right\vert ).
\end{align*}
The first estimate is due to the triangle inequality in Proposition
\ref{Thm-Dual-Norm}, the second estimate follows using Jensen's inequality,
the last inequality is immediate. Combining these inequalities we conclude
\[
L_{1}\leq\sigma^{2}\left\Vert Z\right\Vert _{\alpha^{-1}\text{-}(q,t)}%
^{2}/Var(\left\vert e\right\vert ).
\]

\end{proof}

\section{Additional Materials\label{Sec-Addidtional}}

In this Section, we will provide the proofs for DRO
representation and asymptotic result for logistic regression, which were
discussed in Theorem \ref{Thm-Dro-GLasso-Classification} and Theorem
\ref{Thm-RWPI-Asymptotic}, in Section \ref{Sec_Append_LR_DRO} and Section
\ref{Sect_Appendix_RWP-LR}. In addition, we will provide the results under the
high dimension setting for linear regression, where the number of predictors
growth with the sample size, as a generalization of Theorem
\ref{Thm-RWPI-Asymptotic_Lin}, which we proved in Section \ref{Sec-Supply-Highdim}.

\subsection{Proof of DRO for Logistic Regression\label{Sec_Append_LR_DRO}}

\begin{proof}
[Proof for Theorem \ref{Thm-Dro-GLasso-Classification}]By applying strong
duality results for semi-infinity linear programming problem in
\citeapos{blanchet2016robust}, we can write
the worst case expected loss function as,
\begin{align*}
&  \quad\sup_{P:D_{c}\left(  P,P_{n}\right)  \leq\delta}\mathbb{E}_{P}\left[
\log\left(  1+\exp\left(  -Y\beta^{T}X\right)  \right)  \right]  \\
&  =\min_{\gamma\geq0}\left\{  \gamma\delta-\frac{1}{n}\sum_{i=1}^{n}\sup
_{u}\left\{  \log\left(  1+\exp\left(  -Y_{i}\beta^{T}u\right)  \right)
-\gamma\left\Vert X_{i}-u\right\Vert _{\alpha^{-1}\text{-}(q,t)}\right\}
\right\}  .
\end{align*}
For each $i$, we can apply Lemma 1 in
\citeapos{shafieezadeh-abadeh_distributionally_2015} and the dual norm result in
Proposition \ref{Thm-Dual-Norm} to deal with the inner optimization problem.
It gives us,
\begin{align*}
&  \sup_{u}\left\{  \log\left(  1+\exp\left(  -Y_{i}\beta^{T}u\right)
\right)  -\gamma\left\Vert X_{i}-u\right\Vert _{\alpha^{-1}\text{-}%
(q,t)}\right\}  \\
= &  \left\{
\begin{array}
[c]{ccc}%
\log\left(  1+\exp\left(  -Y_{i}\beta^{T}X_{i}\right)  \right)   & \text{if} &
\left\Vert \beta\right\Vert _{\alpha\text{-}(p,s)}\leq\gamma,\\
\infty & \text{if} & \left\Vert \beta\right\Vert _{\alpha\text{-}(p,s)}%
>\gamma.
\end{array}
\right.
\end{align*}
Moreover, since the outer player wishes to minimize, $\gamma$ will be chosen
to satisfy $\gamma\geq\left\Vert \beta\right\Vert _{\alpha\text{-}(p,s)}$. We
then conclude
\begin{align*}
&  \min_{\gamma\geq0}\left\{  \gamma\delta-\frac{1}{n}\sum_{i=1}^{n}\sup
_{u}\left\{  \log\left(  1+\exp\left(  -Y_{i}\beta^{T}u\right)  \right)
-\gamma\left\Vert X_{i}-u\right\Vert _{\alpha^{-1}\text{-}(q,t)}\right\}
\right\}  \\
&  =\min_{\gamma\geq\left\Vert \beta\right\Vert _{\alpha\text{-}(p,s)}%
}\left\{  \delta\gamma+\frac{1}{n}\sum_{i=1}^{n}\log\left(  1+\exp\left(
-Y_{i}\beta^{T}X_{i}\right)  \right)  \right\}  \\
&  =\frac{1}{n}\sum_{i=1}^{n}\log\left(  1+\exp\left(  -Y_{i}\beta^{T}%
X_{i}\right)  \right)  +\delta\left\Vert \beta\right\Vert _{\alpha
\text{-}(p,s)},
\end{align*}
where the last equality is obtained by noting that the objective function is
continuous and monotone increasing in $\gamma$, thus $\gamma=\left\Vert
\beta\right\Vert _{\alpha\text{-}(p,s)}$ is optimal. Hence, we conclude the
DRO formulation for GR-Lasso logistic regression.
\end{proof}

\subsection{Proof of Optimal Selection of Regularization for Logistic
Regression\label{Sect_Appendix_RWP-LR}}

\begin{proof}
[Proof of Theorem \ref{Thm-RWPI-Asymptotic}]We can apply strong duality result
for semi-infinite linear programming problem in Section B of
\citeapos{blanchet2016robust}, and write the scaled RWP function evaluated at
$\beta^{\ast}$ in the dual form as,
\[
\sqrt{n}R_{n}\left(  \beta^{\ast}\right)  =\max_{\zeta}\left\{  \zeta^{T}%
Z_{n}-\mathbb{E}_{P_{n}}\phi\left(  X,Y,\beta^{\ast},\zeta\right)  \right\}  ,
\]
where $Z_{n}=\frac{1}{n}\sum_{i}^{n}\frac{Y_{i}X_{i}}{1+\exp\left(  Y_{i}%
X_{i}^{T}\beta^{\ast}\right)  }$ and
\[
\phi\left(  X,Y,\beta^{\ast},\zeta\right)  =\max_{u}\left\{  Y\zeta^{T}\left(
\frac{X}{1+\exp\left(  YX^{T}\beta^{\ast}\right)  }-\frac{u}{1+\exp\left(
Yu^{T}\beta^{\ast}\right)  }\right)  -\left\Vert X-u\right\Vert _{\alpha
^{-1}\text{-}(q,t)}\right\}  .
\]
We proceed as in our proof of Theorem \ref{Thm-RWPI-Asymptotic_Lin} in this
paper and also adapting the case $\rho=1$ for Theorem 1 in
\citeapos{blanchet2016robust}. We can apply Lemma 2 in \citeapos{blanchet2016robust}
and conclude that the optimizer $\zeta$ can be taken to lie within a compact
set with high probability as $n\rightarrow\infty$. We can combine the uniform
law of large number estimate as in Lemma 3 of \citeapos{blanchet2016robust} and
obtain
\[
\sqrt{n}R_{n}\left(  \beta\right)  =\max_{\zeta}\left\{  \zeta^{T}%
Z_{n}-\mathbb{E}_{P}\phi\left(  X,Y,\beta^{\ast},\zeta\right)  \right\}
+o_{P}(1).
\]
For the optimization problem defining $\phi\left(  \cdot\right)  $, we can
apply results in Lemma 5 in Section A.3 of \citeapos{blanchet2016robust}, we know,
for any choice of $\tilde{\zeta}$, if,
\[
\mathrm{ess}\sup_{X,Y}\left\Vert \tilde{\zeta}^{T}\frac{y\left(  1+\exp\left(
YX^{T}\beta^{\ast}\right)  \right)  I_{d\times d}-XX^{T}}{\left(
1+\exp\left(  YX^{T}\beta^{\ast}\right)  \right)  ^{2}}\right\Vert
_{\alpha\text{-}(p,s)}>1,
\]
we have $\mathbb{E}\left[  \phi\left(  X,Y,\beta^{\ast},\tilde{\zeta}\right)
\right]  =\infty$. Since the outer optimization problem is maximization over
$\zeta$, the player will restrict $\zeta$ within the set $A$, where
\[
A=\left\{  \zeta\in\mathbb{R}^{d}:\mathrm{ess}\sup_{X,Y}\left\Vert \zeta
^{T}\frac{y\left(  1+\exp\left(  YX^{T}\beta^{\ast}\right)  \right)
I_{d\times d}-XX^{T}}{\left(  1+\exp\left(  YX^{T}\beta^{\ast}\right)
\right)  ^{2}}\right\Vert _{\alpha\text{-}(p,s)}\leq1\right\}  .
\]
Moreover, it is easy to calculate, if $\zeta\in A$, we have $\mathbb{E}%
[\phi\left(  X,Y,\beta^{\ast},\zeta\right)  ]=0$, thus we have the scaled RWP
function has the following estimate, as $n\rightarrow\infty$
\[
\sqrt{n}R_{n}\left(  \beta\right)  =\max_{\zeta\in A}\zeta^{T}Z_{n}+o_{P}(1).
\]
Letting $n\rightarrow\infty$, we obtain the exact asymptotic result.
\bigskip\newline For the stochastic upper bound, let us recall for the
definition of the set $A$ and consider the following estimate
\begin{align*}
&  \left\Vert \zeta^{T}\frac{y\left(  1+\exp\left(  YX^{T}\beta^{\ast}\right)
\right)  I_{d\times d}-XX^{T}}{\left(  1+\exp\left(  YX^{T}\beta^{\ast
}\right)  \right)  ^{2}}\right\Vert _{\alpha\text{-}(p,s)}\\
\geq &  \left\Vert \frac{Y\zeta}{1+\exp\left(  Y\beta^{\ast\;T}X\right)
}\right\Vert _{\alpha\text{-}(p,s)}-\left\Vert \frac{\zeta^{T}X\beta^{\ast}%
}{\left(  1+\exp\left(  Y\beta^{\ast\;T}X\right)  \right)  ^{2}}\right\Vert
_{\alpha\text{-}(p,s)}\\
\geq &  \left(  \frac{1}{1+\exp\left(  Y\beta^{\ast\;T}X\right)  }%
-\frac{\left\Vert X\right\Vert _{\alpha^{-1}\text{-}(q,t)}\left\Vert
\beta^{\ast}\right\Vert _{\alpha\text{-}(p,s)}}{\left(  1+\exp\left(
Y\beta^{\ast\;T}X\right)  \right)  \left(  1+\exp\left(  -Y\beta^{\ast
\;T}X\right)  \right)  }\right)  \left\Vert \zeta\right\Vert _{\alpha
\text{-}(p,s)}.
\end{align*}
The first inequality is due to application of triangle inequality in
Proposition \ref{Thm-Dual-Norm}, while the second estimate follows from
H\"{o}lder's inequality and $Y\in\{-1,+1\}$. Since we assume positive
probability density for the predictor $X$, we can argue that, if $\left\Vert
\zeta\right\Vert _{\alpha\text{-}(p,s)}=\left(  1-\epsilon\right)  ^{-2}>1$
and $\epsilon>0$ is chosen arbitrarily small, we can conclude from the above
estimate that, we have
\[
\left\Vert \zeta^{T}\frac{y\left(  1+\exp\left(  YX^{T}\beta^{\ast}\right)
\right)  I_{d\times d}-XX^{T}}{\left(  1+\exp\left(  YX^{T}\beta^{\ast
}\right)  \right)  ^{2}}\right\Vert _{\alpha\text{-}(p,s)}>1.
\]
Thus, we proved the claim that $A\subset\left\{  \zeta,\left\Vert
\zeta\right\Vert _{\alpha\text{-}(p,s)}\leq1\right\}  $. The stochastic upper
bound is derived by replacing $A$ by $\left\{  \zeta,\left\Vert \zeta
\right\Vert _{\alpha\text{-}(p,s)}\leq1\right\}  $, i.e.
\[
L_{3}=\sup_{\zeta\in A}\zeta^{T}Z\leq\sup_{\left\Vert \zeta\right\Vert
_{\alpha\text{-}(p,s)}\leq1}\zeta^{T}Z=\left\Vert Z\right\Vert _{\alpha
^{-1}\text{-}(q,t)},
\]
where the final estimation is due to dual norm structure in Proposition
\ref{Thm-Dual-Norm}. Since we know, $\frac{1}{1+\exp\left(  YX^{T}%
\beta\right)  }\leq1$, it is easy to argue, $Var(\tilde{Z})-Var(Z)$ is
positive semidefinite, thus, we know $\left\Vert Z\right\Vert _{\alpha
^{-1}\text{-}(q,t)}$ is stochastic dominated by $L_{4}:=\left\Vert \tilde
{Z}\right\Vert _{\alpha^{-1}\text{-}(q,t)}$. Hence, we obtain $L_{3}\leq
L_{4}$. 
\end{proof}

\subsection{Technical Results for Optimal Regularization in GSRL for High
Dimensional Linear Regression\label{Sec-Supply-Highdim}}

We conclude the section by exploring the behavior of the
optimal distributional uncertainty (in the sense of optimality presented in
Section \ref{Sec-Asymptotic}) as the dimension increases. This is an analog of
the high-dimension result for SR-Lasso as Theorem 6 in
\citeapos{blanchet2016robust}.

\begin{theorem}
[RWP Function Asymptotic Results for High-dimension]Suppose that assumptions
in Theorem \ref{Thm-RWPI-Asymptotic_Lin} hold and select $p=2$, $s=1$ let us
write$\tilde{g}^{-1}=\left(  \sqrt{g_{1}},\ldots,\sqrt{g_{\bar{d}}}\right)
^{T}$ (so $\alpha_{j}=\sqrt{g_{j}}$) and ${\tilde{g}}^{-1/2}=\left(
1/\sqrt{g_{1}},\ldots,1/\sqrt{g_{\bar{d}}}\right)  ^{T}$ respectively.
Moreover, let us define $C\left(  n,d\right)  $
\[
C\left(  n,d\right)  =\frac{\mathbb{E}\left\Vert X\right\Vert _{\sqrt{\bar{d}%
}\text{-}(2,1)}}{\sqrt{n}}=\frac{\mathbb{E}\left[  \max_{i=1}^{\bar{d}}%
\sqrt{g_{i}}\left\Vert X\left(  G_{i}\right)  \right\Vert _{2}\right]  }%
{\sqrt{n}}.
\]
Assume that largest eigenvalue of $\Sigma$ is of order $o\left(  nC\left(
n,d\right)  ^{2}\right)  $, that $\beta^{\ast}$ satisfies a weak sparsity
condition, namely, $\left\Vert \beta^{\ast}\right\Vert _{\sqrt{\tilde{g}%
}\text{-}(2,1)}=o\left(  1/C\left(  n,d\right)  \right)  $. Then,
\[
nR_{n}\left(  \beta^{\ast}\right)  \lesssim_{D}\frac{\left\Vert Z_{n}%
\right\Vert _{{\tilde{g}}^{-1/2}\text{-}(2,\infty)}}{Var\left(  \left\vert
e\right\vert \right)  },
\]
as $n,d\rightarrow\infty$, where $Z_{n}:=n^{-1/2}\sum_{i=1}^{n}e_{i}X_{i}$.
\end{theorem}

\begin{proof}
For linear regression model with square loss function, the RWP function is
defined as in equation \eqref{Eqn-RWP-Linear}. By considering the cost
function as in Theorem \ref{Thm-Dro-GLasso-Regression} and applying the strong
duality results in the Appendix of \citeapos{blanchet2016robust}, we can write the scaled RWP function in the
dual form as,
\begin{align*}
nR_{n}\left(  \beta^{\ast}\right)  = &  \sup_{\zeta}\big\{-\zeta^{T}Z_{n}\\
&  -\frac{1}{\sqrt{n}}\sum_{i=1}^{n}\sup_{\Delta}\{e_{i}\zeta^{T}%
\Delta-\left(  \beta^{\ast\;T}\Delta\right)  \left(  \zeta^{T}X_{i}\right)
-\left(  \sqrt{n}\left\Vert \Delta\right\Vert _{{\tilde{g}}^{-1/2}%
\text{-}(2,\infty)}^{2}+\left(  \beta^{\ast\;T}\Delta\right)  \left(
\zeta^{T}\Delta\right)  \right)  \}\big\}.
\end{align*}
For each $i-$th inner optimization problem, we can apply H\"{o}lder inequality
in Proposition \ref{Thm-Dual-Norm} for the term $\left(  \beta^{\ast\;T}\Delta\right)  \left(  \zeta^{T}\Delta\right)  $, we have an upper bound
for the scaled RWP function, i.e.
\begin{align*}
nR_{n} &  \left(  \beta^{\ast}\right)  \leq\sup_{\zeta}\big\{-\zeta^{T}Z_{n}\\
&  -\frac{1}{\sqrt{n}}\sum_{i=1}^{n}\sup_{\Delta}\{\left(  e_{i}\zeta-\left(
\zeta^{T}X_{i}\right)  \beta^{\ast}\right)  ^{T}\Delta-\sqrt{n}\left(
1-\frac{\left\Vert \beta^{\ast}\right\Vert _{\tilde{g}^{-1}\text{-}%
(p,s)}\left\Vert \zeta\right\Vert _{\tilde{g}^{-1}\text{-}(p,s)}}{\sqrt{n}%
}\right)  \left\Vert \Delta\right\Vert _{{\tilde{g}}^{-1/2}\text{-}(q,t)}%
^{2}\}\big\}.
\end{align*}
Since the coefficients for each inner optimization problem is negative and we
can get an upper bound for RWP function if we do not fully optimize the inner
optimization problem. For each $i$, let us take $\Delta$ to the direction
satisfying the H\"{o}lder inequality in Property \ref{Thm-Dual-Norm} for the
term $\left(  e_{i}\zeta-\left(  \zeta^{T}X_{i}\right)  \beta^{\ast}\right)
^{T}\Delta$ and only optimize the magnitude of $\Delta$, for simplicity let us
denote $\gamma=\left\Vert \Delta\right\Vert _{{\tilde{g}}^{-1/2}\text{-}%
(q,t)}$. \newline We have,
\begin{align*}
nR_{n} &  \left(  \beta^{\ast}\right)  \leq\sup_{\zeta}\big\{-\zeta^{T}Z_{n}\\
&  -\frac{1}{\sqrt{n}}\sum_{i=1}^{n}\sup_{\gamma}\{\gamma\left\Vert e_{i}%
\zeta-\left(  \zeta^{T}X_{i}\right)  \beta^{\ast}\right\Vert _{\sqrt
{g}\text{-}(p,s)}-\sqrt{n}\left(  1-\frac{\left\Vert \beta^{\ast}\right\Vert
_{\tilde{g}^{-1}\text{-}(p,s)}\left\Vert \zeta\right\Vert _{\sqrt{\tilde{g}%
}\text{-}(p,s)}}{\sqrt{n}}\right)  \gamma^{2}\}\big\}.
\end{align*}
For each inner optimization problem it is of quadratic form in $\gamma$,
especially, when $n$ is large the coefficients for the second order term will
be negative, thus, as $n\rightarrow\infty$, we can solve the inner
optimization problem and obtain,
\begin{align*}
&  nR_{n}\left(  \beta^{\ast}\right)  \\
&  \leq\sup_{\zeta}\big\{-\zeta^{T}Z_{n}-\frac{1}{4\left(  1-{\left\Vert
\beta^{\ast}\right\Vert _{g^{-1}\text{-}(p,s)}\left\Vert \zeta\right\Vert
_{\tilde{g}^{-1}\text{-}(p,s)}}{n}^{-1/2}\right)  }\frac{1}{{n}}\sum
_{i=1}^{n}\left\Vert e_{i}\zeta-\left(  \zeta^{T}X_{i}\right)  \beta_{\ast
}\right\Vert _{\tilde{g}^{-1}\text{-}(p,s)}^{2}\big\}\\
&  =\sup_{a\geq0}\sup_{\zeta:\left\Vert \zeta\right\Vert _{\sqrt{\tilde{g}%
}\text{-}(p,s)}=1}\big\{-a\zeta^{T}Z_{n}-\frac{a^{2}}{4\left(  1-\left\Vert
\beta^{\ast}\right\Vert _{\tilde{g}^{-1}\text{-}(p,s)}an^{-1/2}\right)
}\frac{1}{{n}}\sum_{i=1}^{n}\left\Vert e_{i}\zeta-\left(  \zeta^{T}%
X_{i}\right)  \beta^{\ast}\right\Vert _{\tilde{g}^{-1}\text{-}(p,s)}%
^{2}\big\}.
\end{align*}
The equality above is due to changing to polar coordinate for the ball under
$\tilde{g}^{-1}\text{-}(p,s)$ norm. For the first term, $\zeta^{T}Z_{n}$,
when $\left\Vert \zeta\right\Vert _{\tilde{g}^{-1}\text{-}(p,s)}=1$, we can
apply H\"{o}lder inequality again, i.e. $\left\vert \zeta^{T}Z_{n}\right\vert
\leq\left\Vert Z_{n}\right\Vert _{{\tilde{g}}^{-1/2}\text{-}(q,t)}$. Then,
only the second term in the previous display involves the direction of $\zeta
$, thus we can have
\begin{align*}
nR_{n}\left(  \beta^{\ast}\right)  \leq &  \sup_{a\geq0}\big\{a\left\Vert
Z_{n}\right\Vert _{{\tilde{g}}^{-1/2}\text{-}(q,t)}\\
&  -\frac{a^{2}}{4\left(  1-{\left\Vert \beta^{\ast}\right\Vert _{\sqrt
{g}\text{-}(p,s)}a}{n}^{-1/2}\right)  }\inf_{\zeta:\left\Vert \zeta\right\Vert
_{\tilde{g}^{-1}\text{-}(p,s)}=1}\frac{1}{{n}}\sum_{i=1}^{n}\left\Vert
e_{i}\zeta-\left(  \zeta^{T}X_{i}\right)  \beta^{\ast}\right\Vert _{\sqrt
{g}\text{-}(p,s)}^{2}\big\}.
\end{align*}
By the weak sparsity assumption, we have $\left\Vert \beta^{\ast}\right\Vert
_{\tilde{g}^{-1}\text{-}(p,s)}{n}^{-1/2}\rightarrow0$ as $n\rightarrow
\infty$, the supremum over $a$ is attained at
\[
a_{{\ast}}=\frac{2\left\Vert Z_{n}\right\Vert _{{\tilde{g}}^{-1/2}%
\text{-}(q,t)}}{\inf_{\zeta:\left\Vert \zeta\right\Vert _{\sqrt{\tilde{g}%
}\text{-}(p,s)}=1}\frac{1}{{n}}\sum_{i=1}^{n}\left\Vert e_{i}\zeta-\left(
\zeta^{T}X_{i}\right)  \beta^{\ast}\right\Vert _{g^{-1}\text{-}(p,s)}^{2}%
}+o(1),
\]
as $n\rightarrow\infty$. Therefore, we have the upper bound estimator for the
scaled RWP function,
\begin{equation}
nR_{n}\left(  \beta^{\ast}\right)  \leq\frac{\left\Vert Z_{n}\right\Vert
_{{\tilde{g}}^{-1/2}\text{-}(q,t)}^{2}}{\inf_{\zeta:\left\Vert \zeta
\right\Vert _{\tilde{g}^{-1}\text{-}(p,s)}=1}\frac{1}{{n}}\sum_{i=1}%
^{n}\left\Vert e_{i}\zeta-\left(  \zeta^{T}X_{i}\right)  \beta_{\ast
}\right\Vert _{g^{-1}\text{-}(p,s)}^{2}}+o_{p}(1).\label{Eqn-RWP_upper}%
\end{equation}
To get the final result, we try to find a lower bound for the infimum in the
denominator. For the objective function in the denominator, since we optimize
on the surface $\left\Vert \zeta\right\Vert _{\tilde{g}^{-1}\text{-}%
(p,s)}=1$, and due to the triangle inequality analysis in Proposition
\ref{Thm-Dual-Norm}, we have
\begin{align*}
&  \quad\frac{1}{{n}}\sum_{i=1}^{n}\left\Vert e_{i}\zeta-\left(  \zeta
^{T}X_{i}\right)  \beta^{\ast}\right\Vert _{\tilde{g}^{-1}\text{-}(p,s)}%
^{2}\\
&\geq\frac{1}{{n}}\sum_{i=1}^{n}\left(  \left\vert e_{i}\right\vert
\left\Vert \zeta\right\Vert _{\tilde{g}^{-1}\text{-}(p,s)}-\left\vert
\zeta^{T}X_{i}\right\vert \left\Vert \beta^{\ast}\right\Vert _{\sqrt{\tilde
{g}}\text{-}(p,s)}\right)  ^{2}\\
&  =\frac{1}{{n}}\sum_{i=1}^{n}\left\vert e_{i}\right\vert ^{2}+\left\Vert
\beta^{\ast}\right\Vert _{\tilde{g}^{-1}\text{-}(p,s)}^{2}\frac{1}{{n}}%
\sum_{i=1}^{n}\left\vert \zeta^{T}X_{i}\right\vert ^{2}-2\left\Vert
\beta^{\ast}\right\Vert _{\tilde{g}^{-1}\text{-}(p,s)}\mathbb{E}\left[
\left\vert e_{i}\right\vert \right]  \frac{1}{{n}}\sum_{i=1}^{n}\left\vert
\zeta^{T}X_{i}\right\vert -\epsilon_{n}\left(  \zeta\right)  ,
\end{align*}
where $\epsilon_{n}\left(  \zeta\right)  =2\left\Vert \beta^{\ast}\right\Vert
_{\tilde{g}^{-1}\text{-}(p,s)}\frac{1}{{n}}\sum_{i=1}^{n}\left(  \left\vert
e_{i}\right\vert -\mathbb{E}\left[  \left\vert e_{i}\right\vert \right]
\right)  $. Let us denote the pseudo error to be $\tilde{e}_{i}=\left\vert
e_{i}\right\vert -\mathbb{E}\left[  \left\vert e_{i}\right\vert \right]  $,
which has mean zero and $Var\left[  \tilde{e}_{i}\right]  \leq Var\left[
e_{i}\right]  $. Since $e_{i}$ is independent of $X_{i}$ we have that
\begin{align*}
&  \mathbb{E}\left[  \tilde{e}_{i}\left\vert \zeta^{T}X_{i}\right\vert
\right]  =0,\\
&  Var\left[  \tilde{e}_{i}\left\vert \zeta^{T}X_{i}\right\vert \right]
=Var\left[  \tilde{e}_{i}\right]  \zeta^{T}\Sigma\zeta\leq Var\left[  {e}%
_{i}\right]  \zeta^{T}\Sigma\zeta.
\end{align*}
By our assumptions on the eigenstructure of $\Sigma$, i.e. $\lambda_{\max
}\left(  \Sigma\right)  =o\left(  nC(n,d)^{2}\right)  $, for the case $p=2$
and $s=1$, we have
\[
\sup_{\zeta:\left\Vert \zeta\right\Vert _{\tilde{g}^{-1}\text{-}(2,1)}%
=1}\zeta^{T}\Sigma\zeta\leq\sup_{\zeta:\left\Vert \zeta\right\Vert
_{\tilde{g}^{-1}\text{-}(2,1)}=1}\lambda_{\max}\left(  \Sigma\right)
\left\Vert \zeta\right\Vert _{2}\leq\lambda_{\max}\left(  \Sigma\right)
=o\left(  nC(n,d)^{2}\right)  .
\]
Then, we have the variance of $\frac{1}{{n}}\sum_{i=1}^{n}\left\vert \zeta
^{T}X_{i}\right\vert $ is of order $o\left(  C(n,d)^{2}\right)  $ uniformly on
$\left\Vert \zeta\right\Vert _{\tilde{g}^{-1}\text{-}(p,s)}=1$. Combining
this estimate with the weak sparsity assumption that we have imposed, we have
\[
\epsilon_{n}\left(  \zeta\right)  =o_{p}(1).
\]
Since the estimate is uniform over $\left\Vert \zeta\right\Vert _{\sqrt
{\tilde{g}}\text{-}(2,1)}=1$, we have that for $n$ sufficiently large,
\begin{align*}
&  \frac{1}{{n}}\sum_{i=1}^{n}\left\Vert e_{i}\zeta-\left(  \zeta^{T}%
X_{i}\right)  \beta^{\ast}\right\Vert _{\tilde{g}^{-1}\text{-}(2,1)}^{2}\\
&  \geq\frac{1}{{n}}\sum_{i=1}^{n}\left\vert e_{i}\right\vert ^{2}-\left(
\mathbb{E}\left[  \left\vert e_{i}\right\vert \right]  \right)  ^{2}%
+\inf_{\zeta:\left\Vert \zeta\right\Vert _{\tilde{g}^{-1}\text{-}(2,1)}%
=1}\left(  \left\Vert \beta^{\ast}\right\Vert _{\tilde{g}^{-1}\text{-}%
(2,1)}\frac{1}{{n}}\sum_{i=1}^{n}\left\vert \zeta^{T}X_{i}\right\vert
-\mathbb{E}\left[  \left\vert e_{i}\right\vert \right]  \right)  ^{2}%
+o_{p}(1)\\
&  \geq Var_{n}\left[  \left\vert e_{i}\right\vert \right]  +o_{p}(1).
\end{align*}
Combining the above estimate and equation \eqref{Eqn-RWP_upper}, when $p=q=2$,
$s=1$ and $t=\infty$, we have that
\[
nR_{n}\left(  \beta^{\ast}\right)  \leq\frac{\left\Vert Z_{n}\right\Vert
_{{\tilde{g}}^{-1/2}\text{-}(2,\infty)}^{2}}{Var\left[  \left\vert
e\right\vert \right]  }+o_{p}(1),
\]
as $n\rightarrow\infty$.
\end{proof}

\end{document}